\documentclass[a4paper,12pt,centertags]{amsart}
\usepackage[utf8]{inputenc}
\usepackage[T1]{fontenc}
\usepackage[sc,osf]{mathpazo}
\usepackage[euler-digits]{eulervm}
\usepackage{mathrsfs}
\usepackage{amssymb}
\usepackage{mathtools}
\usepackage{thmtools}
\usepackage[a4paper,centering]{geometry}
\usepackage[pdfdisplaydoctitle,colorlinks,breaklinks,urlcolor=blue,linkcolor=blue,citecolor=blue]{hyperref}

\usepackage{mathscinet}
\def\MRnum#1\empty{#1}
\renewcommand{\MRhref}[2]{%
  \href{http://www.ams.org/mathscinet-getitem?mr=#1}{#2}
}
\renewcommand{\MR}[1]{
  \relax\ifhmode\unskip\space\fi
  \MRhref{\MRnum#1\empty}{\texttt{\Tiny[MR\MRnum#1\empty]}}
}
\newtheorem{theorem}{Theorem}[section]
\newtheorem{lemma}[theorem]{Lemma}
\newtheorem{proposition}[theorem]{Proposition}
\newtheorem{corollary}[theorem]{Corollary}
\theoremstyle{definition}
\newtheorem{definition}[theorem]{Definition}
\theoremstyle{remark}
\newtheorem{remark}[theorem]{Remark}

\numberwithin{equation}{section}

\DeclareMathOperator{\supp}{supp}
\DeclareMathOperator{\Law}{Law}
\DeclareMathOperator{\Cov}{Cov}
\DeclareMathOperator{\Var}{Var}
\newcommand{\N}{\mathbb{N}}
\newcommand{\R}{\mathbb{R}}
\newcommand{\C}{\mathbb{C}}

\newcommand{\Sb}{\mathbb{S}}
\newcommand{\Dc}{\mathcal{D}}

\newcommand{\E}{\mathbb{E}}
\newcommand{\Prob}{\mathbb{P}}

\newcommand{\1}{\mathbb 1}
\usepackage{xspace}

\newcommand{\eps}{\epsilon}
\newcommand{\les}{\lesssim}
\newcommand{\ph}{\varphi}
\newcommand{\sbe}{{\normalfont\textsc{S\kern-0.12em\raise-0.3ex\hbox{B}\kern-0.12em\hbox{E}}}\xspace}
\newcommand{\SBE}{\sbe_{q}^{\alpha,p}}
\newcommand{\PM}{\textsc{P\kern-0.12em\raise-0.3ex\hbox{M}}}
\newcommand{\var}[1]{{#1-\text{var}}}
\newcommand{\Is}{\mathscr{I}}
\newcommand{\Js}{\mathscr{J}}
\newcommand{\Tc}{\mathcal{T}}
\newcommand{\cerchiato}[1]{\text{\textcircled{\small #1}}}
\begin{document}
\title[Yet another notion of irregularity]{Yet another notion of irregularity through small ball estimates}
  \author[M. Romito]{Marco Romito}
    \address{Dipartimento di Matematica, Universit\`a di Pisa, Largo Bruno Pontecorvo 5, I--56127 Pisa, Italia }
    \email{\href{mailto:marco.romito@unipi.it}{marco.romito@unipi.it}}
    \urladdr{\url{http://people.dm.unipi.it/romito}}
  \author[L. Tolomeo]{Leonardo Tolomeo}
    \address{Mathematisches Institut, Universit\"at Bonn, Endenicher Allee 60, D--53115 Bonn, Nordrhein-Westfalen, Deutschland}
    \email{\href{mailto:tolomeo@math.uni-bonn.de}{tolomeo@math.uni-bonn.de}}
    \urladdr{\url{https://www.math.uni-bonn.de/~tolomeo}}
  \date{July 6, 2022}
  \begin{abstract}
    We introduce a new notion of irregularity of paths, in terms
    of control of growth of the size of small balls by means of
    the occupation measure of the path. This notion ensures
    Besov regularity of the occupation measure and thus
    extends the analysis of Catellier and Gubinelli
    \cite{CatGub2016} to general Besov spaces.
    On stochastic processes this notion is granted by suitable
    properties of local non-determinism.
  \end{abstract}
\maketitle
\tableofcontents

\section{Introduction}

The seminal paper \cite{CatGub2016} introduces the
idea of \emph{noiseless regularization by noise},
further developed in \cite{GalGub2020a,GalGub2020b},
that considers regularization by noise from the point
of view of generic perturbations, without a prescribed
probabilistic distribution on the driving function.
The motivating problem is to solve the stochastic differential
equation,
\begin{equation}\label{e:eq}
  \begin{cases}
    \dot X_t
      = b(t,X_t) + \dot\omega_t,\\
    X_0
      = x,
  \end{cases}
\end{equation}
in the case of singular drift $b$. Regularization
by noise phenomena concern the effects
of perturbation, mainly due to
irregularity, to improve results
of existence/uniqueness/stability
of solutions of the above \eqref{e:eq}.

Regularization by noise has a long history, starting from
\cite{Zvo1974,Ver1981}, and is mainly
devoted to Brownian perturbation, in finite
(see for instance \cite{KryRoc2005,FlaGubPri2010,FlaIssRus2017,Cha2017,
Cha2018,BecFlaGubMau2019,ChoGes2019}) and infinite dimension
(for instance \cite{DapFla2010,DapFlaRocVer2016,ButMyt2019}).
A recent review of probabilistic
regularization by noise is \cite{Fla2011}.
Regularization by noise has then found
application for a wider class of driving
processes (mainly fractional Brownian motion,
or Lévy stable processes), see for
instance \cite{NuaOuk2002,NuaOuk2003,HuNua2007,Pri2012,
AmiBanPro2020,AthButMyt2020,HarLin2020,GalHar2020,KrePer2022,
Ger2022,BecHof2022} and the forthcoming
\cite{GalGer2022}, or for distribution dependent
SDEs \cite{GalHarMay2021,GalHarMay2022}.

Davie in \cite{Dav2007} (see also \cite{Fla2011,Sha2016})
introduced the notion of ``path-by-path'' uniqueness,
and with it a path-wise strategy to regularization
by noise. It is \cite{CatGub2016} that has
initiated an approach to regularization by noise
based on Young's integral and on a notion of
irregularity of paths that does not depend
on an underlying probabilistic structure.
The approach has been further developed
in \cite{GalGub2020a,GalGub2020b}
(see also \cite{GalHar2020,HarPer2020}).
\cite{CatGub2016} introduces
the notion of $(\rho,\gamma)$-irregularity,
which corresponds to space-time regularity
of the Fourier transform of the occupation
measure
\begin{equation}\label{e:om}
  \mu^\omega_{s,t}(\cdot)
    = \int_s^t \1_{\omega_r}(\cdot)\,dr,
\end{equation}
of the path $\omega$. Indeed,
with the position $\theta_t=X_t-\omega_t$,
\eqref{e:eq} reads as
$\dot\theta_t=b(t,\omega_t+\theta_t)$,
or more precisely as
\[
  \theta_t
    = x + \int_0^t b(s,\omega_s+\theta_s)\,ds.
\]
The above integral can be understood in terms
of the occupation measure
by means of a suitable Young's integral.
In this respect, \cite{CatGub2016}
proves that the occupation measure of
a $(\rho,\gamma)$-irregular path
is H\"older in time with values
in \emph{Fourier-Lebesgue} spaces.

Noiseless regularization by noise
can be developed from different points
of view. In the perspective of \eqref{e:eq},
regularization by noise can be
derived for a given drift $b$ and
\emph{generic} perturbations $\omega$,
or for generic drifts and a given
perturbation $\omega$.
A key feature of this second approach
is that, for a given irregular perturbation
with regular occupation measure,
regularization results hold for large
classes of singular drifts,
and is particularly suited for application
to PDEs \cite{GalGub2021}.
The basic tool of this analysis
is the \emph{averaging} operator
$T^\omega$,
that can be recast as a convolution
with the occupation measure,
\[
  T^\omega_{s,t}b(x)
    = \int_s^t b(x+\omega_r)\,dr
    = b\star\mu_{s,t}^{-\omega}.
\]
Mapping properties of the averaging
operator can be achieved by
means pf regularity of the
occupation measure.

A fundamental review on occupation
measure is \cite{GemHor1980}.
A recent review on nonlinear Young integration
can be found in \cite{Gal2021}.

\medskip

In this paper we introduce a new notion
of irregularity, based on control
of the growth of the size of
small balls measured through the
occupation measure.
Our new notion (briefly denoted by \sbe,
a contraction for \emph{small ball
estimate}) is comparable
with the $(\rho,\gamma)$-irregularity
of \cite{CatGub2016}, and presents
a series of features: 
the \sbe property is invariant by
re-parametrization by suitable
bi-Lipschitz maps, moreover
is stable by regular
perturbation. Above all,
\sbe regularity is essentially equivalent to 
Besov regularity of the
occupation measure,
and this allows us to extend
the validity of \cite{CatGub2016}
for \eqref{e:eq} to drift
in Besov spaces, instead of
Fourier-Lebesgue spaces as in
\cite{CatGub2016}.
When we turn to stochastic
processes, the \sbe property
is granted by simple
conditions on the law of
increments of the perturbation
$\omega$, which are known
in the literature as
\emph{local non-determinism}
\cite{Ber1973,GalGub2020b}.
The connection of small ball
estimates with irregularity
of paths is not unexpected.
For instance, small ball
estimates of stochastic processes
yield their $\theta$-roughness in the
sense of \cite{HaiPil2013,FriHai2014} (and
in some way also vice versa, see \cite{HocRomTol2022}).
Likewise, Ess\'een concentration inequality
(see for instance \cite{TaoVu2006}) provides
a bound of mall balls by means of the Fourier
transform of the occupation measure,
a quantity of interest in the
$(\rho,\gamma)$-irregularity of
\cite{CatGub2016}.
\medskip

The paper is organised as follows. In \autoref{s:sbe} we introduce our
new notion, by means of a regularity property on measures.
We then prove that, at least on compactly supported measures, this new notion is
comparable to the scale of Besov spaces. \autoref{s:sbeprops} contains the
proofs that the new notion, when applied to occupation measures, is
invariant by re-parametrization by suitable bi-Lipschitz maps, and
stable by perturbations. In \autoref{s:sbestochproc} we turn to
stochastic processes, and prove \sbe regularity of the occupation
measure of a process under a very simple and slick condition
on the density of increments of the process. This condition,
once read on Gaussian processes, becomes even simpler
and very elementary.
In \autoref{s:catgub} we obtain a higher time regularity, possibly
at the price of a lower space regularity, under a stronger assumption
on the joint laws of time increments. This condition is very close
to Berman's definition of local non-determinism for general
stochastic processes \cite{Ber1983}.
Even though we could rely on the general theory of non-linear
Young's integration developed in \cite{CatGub2016} (see also
\cite{Gal2021}) for H\"older continuous occupation
measures, we give a limited version of the theory,
targeted at equations with additive perturbation,
that works in spaces of variations.
Finally in \autoref{s:examples} we apply the theory
developed so far to solutions to one-dimensional
equations driven by Brownian motion, and to
equations driven by fractional Brownian motion.

\subsection{Notation}\label{s:notation}

If $E$ is a normed space, we shall always denote by $\|\cdot\|_E$ its norm.
We denote by $S(\R^d)$ the Schwarz space, by $B^s_{p,q}=B^s_{p,q}(\R^d)$
the standard Besov spaces on $\R^d$, by $L^p(A,\mu)$ the space of
functions on $A$ whose $p$-power is  integrable with respect to
the measure $\mu$. If $E$ is  normed space, the space $V^p(a,b;E)$
is the set of functions $f:[a,b]\to E$ such that
\begin{equation}\label{e:vp}
  \|f\|_{V^p(a,b;E)}
    := \sup_{a= t_0 < t_1 < \dotsb < t_n = b} \sum_{h = 1}^n \| f(t_n) - f(t_{n-1}) \|_E^p
    < + \infty,
\end{equation}
where the supremum is extended over all finite partitions of $[a,b]$.
The space $C^\var{p}$ is the space of all continuous functions
that are in $V^p(a,b;E)$, with semi-norm
$\|\cdot\|_{C^\var{p}(a,b;E)}=\|\cdot\|_{V^p(a,b;E)}$.
Notice that $\|\cdot\|_{C^\var{p}(a,b;E)}$ is a semi-norm,
and it can be made a norm by adding $\|\cdot\|_\infty$
or the value at one single point. It becomes
a norm on the subspace of $C^\var{p}(a,b;E)$
of functions with given value at one single point.

We denote by $M(\R^d)$ the space of finite measures on $\R^d$,
and by $\delta_x$ the delta measure on $x$.

Finally, $\star$ denotes convolution, $\Sb^d$ is the $d$-sphere,
$|\cdot|$ is the $\R^d$-norm or the Lebesgue measure (use follows
obviously from the context), $\text{Law}(X)$ is the law of the
random object $X$. We will sometimes write $x_{1:n}$ to denote
the vector $(x_1,x_2,\dots,x_n)$. Likewise, in integrals,
$dx_{1:n}$ is the differential $dx_1\,dx_2\dots dx_n$.
We shall use the symbol $\lesssim$ for inequalities
up to numerical factors, which may change from line to line
and that do not depend on the main quantities of the problem.

\subsection{Acknowledgments}

The first author wishes to dedicate this
work to Emilio, he has arrived whilst
the paper was being completed.

The second author was supported by the Deutsche
Forschungsgemeinschaft (DFG, German Research Foundation) under Germany's Excellence
Strategy-EXC-2047/1-390685813, through the Collaborative Research Centre (CRC) 1060. 

\section{Yet another notion of irregularity}\label{s:sbe}

For a measure $\mu \in M(\R^d)$, let $F_\mu : \R^+ \times \R^d \to \R$ be given by 
\begin{equation}\label{Fmudef}
F_\mu(r, y) := \mu(\{ x\in\R^d : |x-y| \le r \}).
\end{equation}
Notice that $F_\mu(r,y)$ is simply the measure of the ball with radius $r$ and centre in $y$.
For a function $F:\R\to\R$, define recursively the operators
$\Delta_k$ by the formula 
\begin{equation} \label{deltakdef}
\begin{cases}
\Delta_0 F(r) = F(r) - F(r/2)\\
\Delta_{k+1} F(r) = \Delta_k F(r) - 2^{k+1}\Delta_k F(r/2).
\end{cases}
\end{equation}
It is easy to check that, if $p_k(r)$ is a polynomial of degree $\le k$, then 
\begin{equation}\label{deltakernel}
\Delta_k p_k(r) = 0.
\end{equation}
Define finally, for $y \in \R^d$, the left translation $\tau_y$ by
\[
\tau_y F(x) = F(x-y).
\]
\begin{definition}
Given $\alpha>0$, $1\leq p,q\leq\infty$, we say that a compactly
supported measure $\mu\in M(\R^d)$ satisfies a \emph{small ball estimate
of regularity $\alpha$, integrability $p$, and values in $L^q$},
in short $\SBE$, if for $k = \lceil \alpha + d -1 \rceil$,
the quantity
\[
  \| \mu \|_{\SBE}
    \vcentcolon=\| r^{-\alpha-d} (\Delta_k)_r F_\mu(r, y)
      \|_{L^q_yL^p_r(\R^d\times\R^+, dy\otimes\frac{dr}{r})}
    < + \infty
\]
is finite.
\end{definition}
The \emph{small ball estimate} condition we have introduced
is comparable to the scale of Besov spaces. The
proof will be given at page \pageref{pf:besovregularity}.
\begin{theorem} \label{besovregularity}
Let $\alpha>0$, $1\leq p<\infty$, and $1\leq q\leq\infty$.
If $\mu\in\SBE$, then $\mu$ belongs to the Besov space $B^{\alpha}_{q,\infty}$. More precisely, 
\begin{equation} \label{besovestimate}
\|\mu\|_{B^{\alpha}_{q,\infty}} \les \| \mu \|_{\SBE}.
\end{equation}
\end{theorem} 
A converse is also available (the proof is detailed
at page \pageref{pf:BiffSBE}).
\begin{proposition} \label{BiffSBE}
Let $\mu = f dx$ be a compactly supported finite measure such that for some $\alpha > 0$, $\alpha \not \in \N$,
\[
f \in B^{\alpha}_{q,1}.
\]
Then, for every $1 \le p \le \infty$, we have that 
\[
\|\mu\|_{\sbe^{\alpha,p}_{q}} \les \| \mu \|_{B^{\alpha}_{q,1}}.
\]
\end{proposition}
\subsection{Proofs of \autoref{besovregularity} and \autoref{BiffSBE}}

Before moving to the proofs, we will need some preparatory lemmas.
Let $\Delta_k^\ast$ be the adjoint of $\Delta_k$ on $L^2(\R^+, \frac{dr}{r})$,
then
\[
\begin{cases}
\Delta_0^\ast F(r) = F(x) - F(2r), \\
\Delta_{k+1}^\ast F(r) = \Delta_k^\ast F(r) - 2^{k+1} \Delta_k^\ast F(2r).
\end{cases}
\]
\begin{lemma}
  Let $\ph: \R^+ \to \C$ be a Schwartz function. Then for every
  $k\geq0$ and $ r\in\R^+$, 
  \begin{equation} \label{eqn:deltaksum}
    \ph(r)
      = \sum_{h=0}^{+\infty} c_{h,k} \Delta_k^\ast \ph(2^hr),
  \end{equation}
  where
  \begin{equation}\label{eqn:chkdef}
    c_{h,k}
      = \sum_{h_0+\dots+h_k=h}\prod_{j=1}^k 2^{j h_j}
      \leq 2^k\cdot 2^{hk}.
  \end{equation}
\end{lemma}
\begin{proof}
  We start by proving the inequality $c_{h,k}\leq C(k)2^{hk}$.
  We clearly have that $c_{h,0} = 1$, so $C(0) = 1$.
  By proceeding inductively, we have 
  \begin{align*}
    c_{h,k+1}
      &= \sum_{h_0 + \dots + h_{k+1} = h} \prod_{j=0}^{k+1} 2^{jh_j} \\
      &= \sum_{h_0=0}^h\Bigl(2^{h-h_0}\sum_{h_1 + \dotsb + h_{k+1} = h - h_0}
        \prod_{j=0}^{k+1} 2^{(j-1)h_j}\Bigr)\\
      &=\sum_{h_0=0}^h 2^{h-h_0}c_{h-h_0,k}\\
      \intertext{thus, by the induction assumption,}
      &\leq C(k) \sum_{h_0=0}^h 2^{h- h_0}2^{(h-h_0)k}\\
      &\leq C(k+1) 2^{h(k+1)},
  \end{align*}
  if we set
  \[
    C(k+1)
      = C(k)\sum_{h_0=0}^h 2^{-h_0(k+1)}.
  \]
  Since the latter sum is bounded from above by $2$, the
  inequality follows.
Therefore, in order to obtain \eqref{eqn:chkdef}, it is enough to take 
\[
C(k+1) = C(k) \sup_{h \in \N} \sum_{h_0 = 0}^h  2^{-h_0(k+1)} \le 2 C(k) < +\infty.
\]
We now move to the proof of formula \eqref{eqn:deltaksum}.
We first notice that since $\ph$ is a Schwartz function,
then $\Delta_k^\ast \ph$ is a Schwartz function as well
for every $k$. Hence the summation in \eqref{eqn:deltaksum}
converges absolutely, in view of the estimate \eqref{eqn:chkdef}.
Moreover, $2^{(k+1)h} \Delta_k^\ast \ph(2^h r) \to 0$ as $h\to\infty$,
so by a telescoping sum argument,
\[
\ph(r) = \sum_{h=0}^{+\infty} \Delta_0^\ast \ph(2^h r),
\]
which is \eqref{eqn:deltaksum} in the case $k = 0$, and similarly
\[
\Delta_k^\ast \ph(r) = \sum_{h=0}^{+\infty}  2^{(k+1)h} \Delta_{k+1}^\ast \ph(2^h r).
\]
Therefore, proceeding inductively, 
\begin{align*}
  \ph(r)
    &= \sum_{h=0}^{+\infty} c_{h,k} \Delta_k^\ast \ph(2^hr) \\
    &= \sum_{h=0}^{+\infty} c_{h,k} \sum_{h_{k+1} = 0}^{+\infty}
      2^{(k+1)h_{k+1}} \Delta_{k+1}^\ast \ph(2^{h+h_{k+1}} r) \\
    &= \sum_{h=0}^{+\infty} \sum_{h_0 + \dotsb + h_k = h} \sum_{h_{k+1} = 0}^{+\infty}
       \prod_{j=0}^k 2^{jh_j} 2^{(k+1)h_{k+1}} \Delta_{k+1}^\ast \ph(2^{h+h_{k+1}} r) \\
    &= \sum_{h'=0}^{+\infty} \sum_{h_{k+1} = 0}^{+\infty} \sum_{h_0 + \dotsb + h_k = h' - h_{k+1}}
     \prod_{j=0}^k
        2^{jh_j} 2^{(k+1)h_{k+1}} \Delta_{k+1}^\ast \ph(2^{h'} r)\\
    &= \sum_{h'=0}^{+\infty} \sum_{h_0 + \dotsb + h_{k+1} = h' }
      \prod_{j=0}^{k+1} 2^{jh_j} \Delta_{k+1}^\ast \ph(2^{h'} r)\\
    &= \sum_{h'=0}^{+\infty} c_{h',k+1} \Delta_{k+1}^\ast \ph(2^{h'} r),
\end{align*}
which shows \eqref{eqn:deltaksum}. 
\end{proof}

\begin{lemma}
  Let $\ph$ be a Schwartz function, $N\geq 1$, $k\geq0$,
  and let $c_{h,k}$ be as in \eqref{eqn:chkdef}. Then
  for every $\alpha > \min(k,1)$ and $1\leq p'\leq\infty$, 
  \begin{equation} \label{eqn:0div}
    \sum_{h=0}^{+\infty} c_{h,k} N \| \ph(2^h N r) |r|^\alpha  \|_{L^{p'}(\frac{dr}{r})} 
      \les N^{-\alpha}.
  \end{equation}
\end{lemma}
\begin{proof}
Since $\ph$ is a Schwartz function and $\alpha > 1$, then
$|r|^\alpha\ph \in L^{p'}(\tfrac{dr}{r})$. Hence, by a simple change of variable,
\[
\| \ph(2^h N r) |r|^\alpha \|_{L^{p'}(\frac{dr}{r})} \sim 2^{-h\alpha} N^{-\alpha}.
\]
Therefore,
\begin{multline*}
  \big \| \sum_{h=0}^{+\infty} c_{h,k} N \ph(2^h N r) |r|^\alpha \big \|_{L^{p'}} 
    \leq \sum_{h=0}^{+\infty} c_{h,k} N \| \ph(2^h N r) |r|^\alpha \|_{L^{p'}}\\
    \les  \sum_{h=0}^{+\infty}   N^{-\alpha} 2^{-h(\alpha - k)}
    \les N^{-\alpha},
\end{multline*}
 by \eqref{eqn:chkdef}.
\end{proof}

\begin{proof}[Proof of \autoref{besovregularity}]\phantomsection\label{pf:besovregularity}
Let $\ph_N = N^d \ph(N \cdot) $ be the Littlewood-Paley projector on frequencies $\sim N$.
The estimate \eqref{besovestimate} is equivalent to 
\begin{equation}\label{besov2}
\Big \| \int_{\R^d} \ph_N(x) \mu(x-y) dx  \Big\|_{L^q_y}
  \les N^{-\alpha}\| \mu \|_{\SBE}. 
\end{equation}
We will first show the above inequality in the case $\mu \in C^\infty(\R^d)$, with
a pre-factor independent from $\mu$.

Recall that $\ph$ is a radial function. Let $\partial_r \ph$ be the
derivative of $\ph$ in the radial direction. With a slight abuse of
notation, we denote $\partial_r \ph(r) := \partial_r \ph(x)$, for any
$x$ with $|x| = r$. Since $\ph$ is a Schwartz function, we have that 
\[
\ph(x) = - \int_0^\infty \partial_r \ph(r) \1_{[0,r]}(|x|)\,dr.
\]
Therefore,
\begin{align*}
\int_{\R^d} \ph_N(x) \mu(x-y) dx
  &= \int_{\R^d} \ph_N(x-y) \mu(x) dx \\
  &= - \int_0^\infty \int_{\R^d} \partial_r \ph_N(r) \1_{[0,r]}(|x-y|)\mu(x)\,dx\,dr\\
  &= - \int_0^\infty \partial_r \ph_N(r) F_\mu(r,y) d r\\
  &= - \int_0^\infty N^d \cdot Nr (\partial_r \ph)(Nr)F_\mu(r,y) \frac{d r}{r}
\end{align*}
In order to compensate the divergence at $r=0$ that we will find when
estimating formula \eqref{eqn:deltaksum} (see \eqref{e:divergence} below),
we notice that, since $\int \ph_N(r) r^n\,dr = 0$
for every $n \in \N$,  by similar computations,
\begin{align*}
 \int_0^\infty N^d \cdot Nr (\partial_r \ph)(Nr)r^{n+1} \frac{d r}{r} = 0
\end{align*}
for every $n \in \N$. Therefore,
\begin{equation}\label{bes1}
  \int \ph_N(x) \mu(x-y) dx
    = -\int_0^\infty N^d \cdot Nr (\partial_r \ph)(Nr)\Bigl(F_\mu(r,y) - \sum_{j=d}^{k} a_j r^j\Bigr)\,\frac{d r}{r},
\end{equation}
and we choose $a_j = a_j(y) = \frac{1}{j!} F_\mu(\cdot,y)^{(j)}(0) $, so that 
\begin{equation} 
\Big| F_\mu(r,y) - \sum_{j=d}^{k} a_j r^j\Big| \les r^{k+1} \label{taylor}. 
\end{equation}
Set $\psi(r)=r \partial_r \ph$, so that $N^d \cdot Nr (\partial_r \ph)(Nr) = N^d \psi(Nr)=: \psi_N(r)$.
By \eqref{eqn:deltaksum}, 
\[
  \psi_N(r) = \sum_{h=0}^{+\infty} c_{h,k} \Delta_k^\ast \psi_N(2^hr).
\]
Since $\psi_N$ is a Schwartz function, and thus so is
$\Delta_k^\ast \psi_N$, from \eqref{eqn:chkdef} we have that 
\begin{equation}\label{e:divergence}
  \sum_{h=0}^{+\infty} c_{h,k} |\Delta_k^\ast \psi_N(2^hr)|
    \les \frac{r^{-k}}{1 + r^2}.
\end{equation}
Therefore, by dominated convergence, \eqref{bes1}, \eqref{taylor},
\eqref{deltakernel}, H\"older's inequality (we denote by $p'$
the H\"older conjugate exponent of $p$) and \eqref{eqn:0div},
\begin{align*}
  \Big|\int\ph_N(x)\mu(x-y)\,dx\Big|
    &= \Big|\sum_{h=0}^{+\infty} c_{h,k}\int\Delta_k^\ast\psi_N(2^hr)
      \Bigl(F_\mu(r,y) - \sum_{j=d}^{k} a_j r^j\Bigr)\,\frac{dr}{r}\Big|\\
    &=\Big| \sum_{h=0}^{+\infty} c_{h,k} \int \psi_N(2^h r)
      \Delta_k F_\mu(r,y)\,\frac{dr}{r}\Big|\\
    &= \Big|\sum_{h=0}^{+\infty} c_{h,k}\int r^{\alpha+d}\psi_N(2^h r)
      r^{-\alpha-d}\Delta_k F_\mu(r,y)\,\frac{dr}{r}\Big| \\
    &\leq \sum_{h=0}^{+\infty} c_{h,k} \| r^{\alpha+d}\psi_N(2^h r)\|_{L^{p'}_r}
      \| r^{-\alpha-d}\Delta_kF_\mu(r, y) \|_{L^p_r} \\
&\les N^{-\alpha} \| r^{-\alpha - d} \Delta_k F_\mu(r,y) \|_{L^p_r},
\end{align*}
from which \eqref{besov2} follows easily.

We now move to the case of $\mu$ being not necessarily smooth.
Let $\rho \ge 0$ be a smooth function with compact support
and $\int \rho = 1$, and for $\eps > 0$, let 
\[
  \rho_\eps := \eps^{-1} \rho(\eps^{-1} x).
\]
Fix $N \gg 1$, and let $\mu_\eps := \mu \star \rho_\eps$. Recalling that $\mu$ has compact support and $\ph_N$ is a Schwartz function, we have that 
\begin{align*}
  \int_{\R^d} \ph_N(x-y) \bigl(d \mu(x) - \mu_{\eps}(x)dx\bigr)
    &=  \int_{\R^d} (\ph_N - \ph_N \star \rho_\eps)(x-y) d \mu(x)\\
    &\les C(N, \mu) \eps (1 + |y|)^{-(d+1)},
\end{align*}
so $\|\int \ph_N(x-y) (d \mu(x) - \mu_{\eps}(x)dx) \|_{L^q_{y}} \to 0$ as $\eps \to 0$. 
Moreover,
\begin{align*}
\| \mu_\eps \|_{\SBE} & =  \| r^{-\alpha-d} \Delta_k  F_{\mu_\eps} (r,y) \|_{L^q_yL^p_r} \\
& = \| r^{-\alpha-d} (\Delta_k)_r (F_{\mu}\star_y \rho_\eps) (r,y) \|_{L^q_yL^p_r} \\
& \le \Big \| \int   \| r^{-\alpha-d} \Delta_k F_{\mu}(r,y-y_0) \|_{L^p_r} \rho_\eps(y_0) d y_0 \Big \|_{L^q_{y}} \\
& \le  \int   \| r^{-\alpha-d} \Delta_k F_{\mu}(r,y-y_0) \|_{L^q_y L^p_r}\rho_\eps(y_0) d y_0  \\
& = \int   \| r^{-\alpha-d} \Delta_kF_{\mu}(r,y) \|_{L^q_y L^p_r}\rho_\eps(y_0) d y_0  \\
& = \| \mu \|_{\SBE}.
\end{align*}
Therefore, by choosing $\eps$ small enough, 
\begin{align*}
  \lefteqn{\Big \| \int \ph_N(x-y) d \mu(x)  \Big\|_{L^q_y}}\\
    &\le\Big \|  \int \ph_N(x-y) (d \mu(x) - \mu_{\eps}(x)dx) \Big \|_{L^q_{y}}
      +  \Big \| \int \ph_N(x-y) \,d \mu_\eps (x)  \Big\|_{L^q_y} \\
    &\les N^{-\alpha}\| \mu \|_{\SBE},
\end{align*}
which yields \eqref{besov2}.
\end{proof}
\begin{lemma}\label{lem:schwartzinsbe}
Let $\ph \in S(\R^d)$. Then for every $\alpha > 0$, $\alpha \not \in \N$, and for every $1 \le p, q \le \infty$,
$$ \| \ph \|_{\SBE} < \infty. $$
\end{lemma}
\begin{proof}
We begin by estimating $\Delta_k F_\ph(r,y)$. For some coefficients $\{a_n= a_n(y)\}_{n \le k}$ to be chosen later, define 
\[
  g(r,y) :=  F_\ph(r,y) - \sum_{n \le k} a_n r^{n+d}.
\]
From \eqref{deltakernel}, we have that $(\Delta_{k})_r g(r,y) = \Delta_{k} F_\ph(r,y)$.
Moreover, since $\ph \in C^\infty(\R^d)$, from the formula 
\[
F_\ph(r,y) = \int_{|x-y| \le r} \ph(x) d x =  \int_{0}^r \int_{\Sb^{d-1}} \mu(r \cdot \theta + y) d \sigma(\theta) r^{d-1} d r,
\]
we have that $F_\ph(\cdot,y) \in C^\infty([0,+\infty))$ as well, and moreover, for every $h \in \N$,
\begin{equation}\label{e:derestimate}
| F_\ph^{(h)}(\cdot,y)(r)| \lesssim \| \ph \|_{C^{h}(B(y,r))},
\end{equation}
and for every $h \le d$,
\begin{equation}\label{e:dernull}
| F_\ph^{(h)}(\cdot,y)(0)| = 0.
\end{equation}
Therefore, choosing 
\[
a_n(y) := \frac{F_\ph^{(n+d)}(\cdot,y)(0)}{(n+d)!}, 
\]
by \eqref{e:derestimate} and \eqref{e:dernull}, we obtain that 
\[
|g(r,y)| \lesssim r^{d+k+1} \| \ph \|_{C^{k+1}(B(y,r))},
\]
and so we also have the same estimate for $(\Delta_{k})_r g(r,y) = \Delta_{k} F_\ph(r,y)$,
\[
|\Delta_{k} F_\ph(r,y)| \les \| \ph \|_{C^{k+1}(B(y,r))}.
\]
Together with the simple estimate 
\[
|\Delta_{k} F_\ph(r,y)| \les \| \ph \|_{L^1},
\]
by splitting the domain of integration in $\{ r \le \frac {|y| + 1}2\}$ and $\{ r > \frac {|y| + 1}2\}$, we obtain that 
\[
\| r^{-\alpha-d} (\Delta_k)_r F_\mu(r, y)
      \|_{L^p_r(\frac{dr}{r})} \les \big(|y| + 1\big)^{k+1-\alpha}\| \ph \|_{C^{k+1}(B(y,\frac {|y| + 1}2))} + \big(|y| + 1\big)^{-d-\alpha}
\]
Recalling that $\ph \in S(\R^d)$, it follows easily that $\|\ph\|_{\SBE} < \infty$.
\end{proof}

\begin{proof}[Proof of \autoref{BiffSBE}]\phantomsection\label{pf:BiffSBE}
For $N \ge 1$, let $\ph_N(x) := N^d \ph(Nx)$ be the Littlewood-Paley projector, and let $\ph_0 = \sum_{N \le 0} N^{d} \ph(N \cdot)$. Since $f \in B^\alpha_{q,1}$, we can write 
\[
f = \sum_{N\ge 0 \text{ dyadic}} f_N \star \ph_N,
\]
with 
\[
\sum_N N^{\alpha} \| f_N\|_{L^q} \lesssim \|f\|_{B^\alpha_{q,1}}.
\]
By change of variables,
\begin{align*}
F_{\ph_N}(r, y) = \int_{B_1} r^d \ph_N(y+rx) dx 
 = \int_{B_1} r^d N^d \ph(Ny+Nrx) dx 
 = F_{\ph}(Nr, Ny).
\end{align*}
Therefore, we also have $(\Delta_k)_r F_{\ph_N}(r, y) = \big((\Delta_k)_r F_{\ph}\big)(Nr, Ny)$. Thus,
\begin{align*}
\|\ph_N\|_{\sbe^{\alpha,p}_1} &= \| r^{-\alpha-d} (\Delta_k)_r F_{\ph_N}(r, y) \|_{L^1_yL^p_r(\R^d\times\R^+, dy\otimes\frac{dr}{r})} \\
&= \| r^{-\alpha-d} \big((\Delta_k)_r F_{\ph}\big)(Nr, Ny) \|_{L^1_yL^p_r(\R^d\times\R^+, dy\otimes\frac{dr}{r})}\\
&= \big\| N^{\alpha + d} \| \big((\Delta_k)_r F_{\ph}\big)( \cdot, Ny) \|_{L^p_r(\frac{dr}{r})}\big\|_{L^1_y}\\
&= N^{\alpha} \|\ph\|_{\sbe^{\alpha,p}_1},
\end{align*}
which is finite by \autoref{lem:schwartzinsbe}. Similarly, again by
\autoref{lem:schwartzinsbe}, we have that $\|\ph_0\|_{\sbe^{\alpha,p}_1} \les 1$.
Therefore, by Young's inequality,
\begin{align*}
\| f \|_{\SBE} &\le \sum_N \| f_N \star \ph_N \|_{\SBE} \\
& = \sum_N  \| r^{-\alpha-d} (\Delta_k)_r F_{f_N \star \ph_N}(r, y) \|_{L^q_yL^p_r(\R^d\times\R^+, dy\otimes\frac{dr}{r})} \\
& = \sum_N  \| r^{-\alpha-d} \big(f_N \star_y (\Delta_k)_r F_{ \ph_N}\big)(r, y) \|_{L^q_yL^p_r(\R^d\times\R^+, dy\otimes\frac{dr}{r})} \\
& \le \sum_N  \| f_N \|_{L^q} \| r^{-\alpha-d} (\Delta_k)_r F_{ \ph_N}(r, y) \|_{L^1_yL^p_r(\R^d\times\R^+, dy\otimes\frac{dr}{r})} \\
& \les \sum_N \| f_N \|_{L^q} N^{\alpha} \les \| f \|_{B^{\alpha}_{q,1}}.\qedhere
\end{align*}
\end{proof}
\section{\texorpdfstring{\sbe}{SBE} regularity for occupation measures}\label{s:sbeprops}

We wish to use this definition for the occupation measure
(defined in \eqref{e:om}), to measure the irregularity
of paths. First, we give a connection with $(\rho,\gamma)$
irregularity.
\begin{proposition}
  Let $\alpha>0$, $1\leq p,q\leq\infty$, and $\gamma\in(0,1)$.
  Given a path $\omega:[a,b]\to\R^d$, if the occupation measure
  $(\mu^\omega_{a,t})_{t\in[a,b]}$ is $\gamma$-H\"older
  with values in $\SBE$, then $\omega$ is
  $(\alpha,\gamma)$-irregular.  
\end{proposition}
\begin{proof}
  We recall that by a simple remark given in \cite{HocRomTol2022}, a path $\omega$
  is $(\rho,\gamma)$-irregular if and only if $t\mapsto\mu^\omega_{a,\cdot}$
  is $\gamma$-H\"older in time with values in $B^{\rho,\infty}_{\PM}$,
  where $\PM$ is the space of Kahane's pseudo-measures. Since
  $L^1\subset\PM$, we have that $B^\alpha_{1,\infty}\subset B^{\alpha,\infty}_\PM$.
  Since $(\mu^\omega_{a,t})_{t\in[a,b]}$ is compactly supported,
  we have that $\mu^\omega_{a,t} = \rho \mu^\omega_{a,t}$ for every
  $\rho \in C^\infty_c$ such that $\rho \equiv 1 $
  on $\supp(\mu^\omega_{a,t})$. Therefore, 
  \[
    \|\mu^\omega_{a,t}\|_{B^{\alpha}_{1,\infty}}
      = \|\rho \mu^\omega_{a,t}\|_{B^{\alpha}_{1,\infty}}
      \les \|\mu^\omega_{a,t}\|_{B^{\alpha}_{q,\infty}}.
  \]
  Hence by \autoref{besovregularity},
  \[
    \|\mu^\omega_{a,t}\|_{B^{\alpha}_{1,\infty}}
      \les \|\mu^\omega_{a,t}\|_{B^{\alpha}_{q,\infty}}
      \les \|\mu^\omega_{a,t}\|_{\SBE}.\qedhere
  \]
\end{proof}
In view of the negative result presented in \cite{HocRomTol2022}
about the non-existence for H\"older continuous maps
of a condition that ensures $(\rho,\gamma)$
irregularity and is invariant by bi-Lipschitz
reparametrisation, we prove that our notion
of irregularity of paths is invariant by bi-Lipschitz reparametrisations, under an additional $V^p$ condition on the derivative of the reparametrisation.
To this end we recall the definition of spaces
of variations $V^p(a,b;E)$ defined in \autoref{s:notation}.
\begin{proposition}\label{prop:reparametrisation}
  Let $\omega: [a,b] \to \R^d$, and let $r > 1$. Suppose that
  the occupation measure $\mu$ of $\omega$ satisfies 
  \[
    \mu_{a,\cdot}
      \in V^r(\SBE).
  \]
  Let $\ph: [a',b'] \to [a,b]$ be a bi-Lipschitz map with
  ${\ph'} \in V^{r'-\eps}$ for some $\eps > 0$, where $r'$
  is the H\"older conjugate of $r$.
  Then the process $\omega \circ \ph$ has occupation
  measure $\mu^{\ph}$ that satisfies 
  \[
    \mu^\ph_{a',\cdot}
      \in V^r(\SBE),
  \]
  as well.
\end{proposition}
We will prove the above proposition at page \pageref{pf:reparametrisation}.
We additionally give a property of stability with respect to (suitable)
smooth perturbations of paths, proved at page \pageref{pf:shift}.
\begin{proposition}\label{prop:shift}
Let $\omega: [a,b] \to \R^d$. Suppose that for some $r>1$, $1<p<\infty$,
$1 \le q \le \infty$, the occupation measure $\mu$ of $\omega$ satisfies 
\[
  \mu_{a,\cdot} \in V^r(B^{\alpha}_{p,q}).
\]
Let $f:[a,b]\to\R^d$ be a function with $f\in V^{r_1}$
for some $r_1 \ge 1$, let $\omega^f_s := \omega_s + f_s$,
and let $\mu^f$ be its occupation measure. 
Then, for every $\frac{r_1}{r'} < \gamma \le 1$, 
\[
  \mu^f_{a,\cdot} \in V^r(B^{\alpha-\gamma}_{p,q}),
\]
where $r'$ is the H\"older conjugate of $r$. Moreover,
if $f,g \in V^{r_1}$, then for every $\frac{r_1}{r'} < \gamma \le 1$,
\begin{equation}\label{Vrdiffestimate}
\| \mu^f_{a,\cdot} - \mu^g_{a,\cdot} \|_{V^r(B^{\alpha-\gamma-1}_{p,q})} \les \| \mu_{a,\cdot}\|_{V^r(B^{\alpha}_{p,q})} \big( \| f - g \|_{V^{r_1}} + (1 +  \|g\|_{V^{r_1}}) \| f - g \|_{L^\infty} \big)
\end{equation}

\end{proposition}
\subsection{Proofs of reparametrisation and perturbation}

\begin{proof}[Proof of \autoref{prop:reparametrisation}]\phantomsection\label{pf:reparametrisation}
We preliminarily state the following estimate. Let $g: [a,b] \to \R$
be a bounded function in $V^{r'-\eps}$, and set
\[
g\cdot\mu_{s,t} := \int_s^t g(\tau)\delta_{\omega_{\tau}}d \tau.
\]
Then
\begin{equation} \label{rep1}
  \| g\cdot\mu_{a,\cdot} \|_{V^{r}(\SBE)}
    \les \| \mu_{a,\cdot} \|_{V^{r}(\SBE)}  \big(\| g \|_{V^{r'-\eps}} + \| g \|_{L^\infty}\big).
\end{equation}

We postpone the proof of the estimate and use it to prove the proposition.
By a change of variables, we have that 
\[
\mu^\ph_{s,t} = \int_s^{t} \delta_{\omega_{\ph(r)}}\,dr = \int_{\ph(s)}^{\ph(t)} \frac{\delta_{\omega_\tau}}{\ph'(\ph^{-1}\tau)}d \tau. 
\]
Let $g(\tau) := \frac{1}{\ph'(\ph^{-1}\tau)}$. Since $\ph' \in V^{r'-\eps}$, and $\ph$ is bi-Lipschitz, then we also have $g \in V^{r-\eps}$. Therefore, by \eqref{rep1}, 
\[
  \| \mu^\ph_{a,\cdot} \|_{V^{r}(\SBE)}
    \les \| \mu_{a,\cdot} \|_{V^{r}(\SBE)} 
      \Big(\Big\| \frac 1 {\ph' \circ \ph^{-1}} \Big\|_{V^{r'-\eps}}
      + \Big\| \frac {1}{\ph'} \Big\|_{L^\infty}\Big) < +\infty,
\]
which concludes the proof of the main statement.

We turn to the proof of \eqref{rep1}.
We first assume that $g$ is smooth.
If $g$ is constant, this inequality is trivial.
Therefore, we can assume that $g$ is not constant.
Let $N$ be a dyadic number. Define the partitions
$s = t_0^N < t_1^N < \dotsb < t_N^N = t$ so that 
\begin{equation}\label{e:equipartition0}
\begin{cases}
\|g\|_{V^{r'-\eps}([t_{2j}^N,t_{2j+1}^N])}^{r'-\eps} = \|g\|_{V^{r'-\eps}([t_{2j+1}^N,t_{2j+2}^N])}^{r'-\eps},\\
t_{2j}^{2N} = t_j^N.
\end{cases}
\end{equation}
Since 
\begin{multline*}
\|g\|_{V^{r'-\eps}([t_{2j}^N,t_{2j+1}^N])}^{r'-\eps} + \|g\|_{V^{r'-\eps}([t_{2j+1}^N,t_{2j+2}^N])}^{r'-\eps} \\
\le \|g\|_{V^{r'-\eps}([t_{2j}^N,t_{2j+2}^N])}^{r'-\eps} = \|g\|_{V^{r'-\eps}([t_{j}^{N/2},t_{j+1}^{N/2}])}^{r'-\eps},
\end{multline*}
proceeding inductively we have that
\begin{equation}\label{e:equipartition1}
\|g\|_{V^{r'-\eps}([t_j^N,t_{j+1}^N])}^{r'-\eps} \le \frac 1 N \|g\|_{V^{r'-\eps}([s,t])}^{r'-\eps} 
\end{equation}
Moreover, we have that 
\[
g\cdot\mu_{s,t} = \lim_{N \to \infty} (g\cdot\mu_{s,t})^{(N)}
\]
in total variation norm, where
\[
(g\cdot\mu_{s,t})^{(N)} := \sum_{j=0}^{N-1} \mu_{t_j^N, t_{j+1}^N} g(t_j^N).
\]
By \eqref{e:equipartition1}, we have that
\[
  \begin{aligned}
    \sum_{j=0}^{N-1}\big|g(t_{2j+1}^{2N}) - g(t_{2j}^{2N})\big|^{r'}
      &\leq \sum_{j=0}^{N-1} \|g\|_{V^{r'-\eps}(t_{2j}^{2N},t_{2j+1}^{2N})}^\eps
        \big|g(t_{2j+1}^{2N}) - g(t_{2j}^{2N})\big|^{r'-\eps}\\
      &\leq N^{-\frac\eps{r'-\eps}}\|g\|_{V^{r'-\eps}(s,t)}^\eps\sum_{j=0}^{N-1}
        \big|g(t_{2j+1}^{2N}) - g(t_{2j}^{2N})\big|^{r'-\eps}\\
      &\leq N^{-\frac\eps{r'-\eps}}\|g\|_{V^{r'-\eps}(s,t)}^{r'}
  \end{aligned}
\]
thus,
\[
  \begin{aligned}
    \lefteqn{\|(g\cdot\mu_{s,t})^{(2N)} - (g\cdot\mu_{s,t})^{(N)}\|_{\SBE}}\\
      &= \Big\| \sum_{j=0}^{N-1} \mu_{t_{2j+1}^{2N}, t_{2(j+1)}^{2N}}
        \big(g(t_{2j+1}^{2N}) - g(t_{2j}^{2N})\big) \Big\|_{\SBE} \\
      &\leq\sum_{j=0}^{N-1} \|\mu_{t_{2j+1}^{2N}, t_{2(j+1)}^{2N}}\|_{\SBE}
        \big|g(t_{2j+1}^{2N}) - g(t_{2j}^{2N})\big|\\
      &\leq\Big(\sum_{j=0}^{N-1} \| \mu_{t_{2j+1}^{2N}, t_{2(j+1)}^{2N}}\|_{\SBE}^r\Big)^\frac 1r
        \Big(\sum_{j=0}^{N-1}\big|g(t_{2j+1}^{2N})- g(t_{2j}^{2N})\big|^{r'}\Big)^\frac 1{r'}\\
      &\leq N^{-\frac \eps{r'(r'-\eps)}}\|\mu\|_{V^r([s,t],\SBE)} \|g\|_{V^{r'-\eps}}.
\end{aligned}
\]
Therefore, by telescopic summation,
\begin{align*}
\|g\cdot\mu_{s,t} \|_{\SBE}
&\lesssim \|(g\cdot\mu_{s,t})^{(1)} \|_{\SBE} + \| \mu \|_{V^r([s,t],\SBE)} \| g \|_{V^{r'-\eps}}  \\
&\lesssim \|\mu_{s,t}\|_{\SBE} \|g\|_{L^\infty} + \| \mu \|_{V^r([s,t],\SBE)} \| g \|_{V^{r'-\eps}},
\end{align*}
from which \eqref{rep1} follows when $g$ is smooth.

Assume now that $g\in V^{r'-\epsilon}$ is bounded.
It is well known that smooth functions are not dense
in spaces of variations, see for instance
\cite[Theorem 5.31]{FriVic2010}, and moreover our
$g$ is not continuous. So we consider
the following regularization $g_\delta$
of $g$ defined as
\[
  g*\rho_\delta(t)
    = \int_\R g(t-s)\rho_\delta(s)\,ds,
\]
with $g$ continuously extended as constant outside $[a,b]$,
and where $\rho_\delta(t)=\delta^{-1}\rho(\delta^{-1}t)$
is a smooth mollifier.
It is elementary to see that
$\|g_\delta\|_{V^{r'-\eps}}\leq \|g\|_{V^{r'-\eps}}$
and $\|g_\delta\|_\infty\leq \|g\|_\infty$.
Moreover,
\[
  \|F_{g\cdot\mu_{s,t}} - F_{g_\delta\cdot\mu_{s,t}}\|_\infty
    \leq \|g - g_\delta\|_{L^1},
\]
therefore by Fatou's lemma,
\[
  \|g\cdot\mu_{a,\cdot}\|_{V^r(\SBE)}
    \leq\liminf_{\delta\to0}\|g_\delta\cdot\mu_{a,\cdot}\|_{V^r(\SBE)},
\]
and from this \eqref{rep1} follows.
\end{proof}
\begin{proof}[Proof of \autoref{prop:shift}]\phantomsection\label{pf:shift}
The proof follows closely the proof of \autoref{prop:reparametrisation}. We want to show that 
\begin{equation} \label{shift1}
\| \mu^f_{a,\cdot} \|_{V^{r}(B^{\alpha- \gamma}_{p,q})} \les \| \mu_{a,\cdot} \|_{V^r(B^\alpha_{p,q})} \|f\|_{V^{r_1}}^\gamma + \| \mu_{a,\cdot} \|_{V^{r}(B^{\alpha- \gamma}_{p,q})}.
\end{equation}
If $f$ is constant, the inequality is trivial, so we can assume that $f$ is not constant.  Define the partitions $s = t_0^N < t_1^N < \dotsb < t_N^N = t$ as in \eqref{e:equipartition0} (with $r'-\eps$ replaced by $r_1$), so that for every $x, y \in [t_j^N,t_{j+1}^N]$, 
\begin{equation}\label{e:equipartition2}
|f(x) - f(y)|^{r_1} \le \frac 1 N \|f\|_{V^{r_1}([s,t])}^{r_1}.
\end{equation}
Recall that $t_{2j}^{2N} = t_{j}^{N}$. Moreover, we have that 
\[
\mu^f_{s,t} = \lim_{N \to \infty} \mu^{f,N}_{s,t}
\]
in (any) Wasserstein norm, where
\[
\mu^{f,N}_{s,t} := \sum_{j=0}^{N-1} \tau_{f_{t_j}}\mu_{t_j^N, t_{j+1}^N}.
\]
For any function $u$ such that its Fourier transform has support
in $[-10N, 10N]$, we have the inequality 
\[
\| \tau_y u - u \|_{L^p} \les |y| N \|u \|_{L^p}.
\]
By interpolation with $\| \tau_y u - u \|_{L^p} \les \| u \|_{L^p}$, this implies
\[
\| \tau_y u - u \|_{L^p} \les |y|^{\gamma} N^{\gamma} \|u \|_{L^p}.
\]
Therefore, for any $x, y \in \R^d$, we have 
\begin{equation}\label{e:holdertranslation}
\| \tau_y u - \tau_x u \|_{B^{\alpha-\gamma}_{p,q}} \les |x-y|^{\gamma} \|u\|_{B^{\alpha}_{p,q}}.
\end{equation}
Set $\epsilon=\gamma r'-r_1$, then by \eqref{e:equipartition2},
\begin{equation} \label{e:gammarvar}
  \begin{aligned}
    \sum_{j=0}^{N-1}\big|f(t_{2j+1}^{2N}) - f(t_{2j}^{2N})\big|^{\gamma r'}
      &\leq \sum_{j=0}^{N-1} (N^{-\frac 1 {r_1}} \|f\|_{V^{r_1}(s,t)})^\eps
        \big|f(t_{2j+1}^{2N}) - f(t_{2j}^{2N})\big|^{r_1}\\
      &\leq N^{-\frac\eps{r_1}}\|f\|_{V^{r_1}(s,t)}^{\gamma r'}.
  \end{aligned}
\end{equation}
The above estimate and \eqref{e:holdertranslation} yield,
\[
  \begin{aligned}
    \|\mu^{f,2N}_{s,t} - \mu^{f,N}_{s,t}\|_{B^{\alpha-\gamma}_{p,q}}
      &=\Big\| \sum_{j=0}^{N-1} (\tau_{f(t_{2j+1}^{2N})} - \tau_{f(t_{2j}^{2N})})
        \mu_{t_{2j+1}^{2N}, t_{2(j+1)}^{2N}} \Big\|_{B^{\alpha-\gamma}_{p,q}} \\
      &\leq\sum_{j=0}^{N-1} \| \mu_{t_{2j+1}^{2N}, t_{2(j+1)}^{2N}}\|_{B^\alpha_{p,q}}
        \big|f(t_{2j+1}^{2N}) - f(t_{2j}^{2N})\big|^\gamma\\
      &\leq\Big(\sum_{j=0}^{N-1}\|\mu_{t_{2j+1}^{2N}, t_{2(j+1)}^{2N}}\|_{B^\alpha_{p,q}}^r \Big)^\frac1r
        \Big(\sum_{j=0}^{N-1}\big|f(t_{2j+1}^{2N}) - f(t_{2j}^{2N})\big|^{\gamma r'}\Big)^\frac 1{r'}\\
      &\leq N^{-\frac{\gamma r' - r_1}{r'r_1}}\|\mu \|_{V^r([s,t],B^\alpha_{p,q})}
        \|f\|_{V^{r_1}}^\gamma.
  \end{aligned}
\]
By telescopic summation, we obtain 
\begin{align*}
\|\mu^{f}_{s,t} \|_{B^{\alpha-\gamma}_{p,q}} 
&\le \|\mu_{s,t} \|_{B^{\alpha-\gamma}_{p,q}} + \sum_{N} \|\mu^{f,2N}_{s,t} - \mu^{f,N}_{s,t}\|_{B^{\alpha-\gamma}_{p,q}} \\
&\les \|\mu_{s,t} \|_{B^{\alpha-\gamma}_{p,q}}  + \| \mu \|_{V^r([s,t],B^\alpha_{p,q})} \| f \|_{V^{r_1}}^\gamma,
\end{align*}
from which \eqref{shift1} follows easily.

Now we move to proving \eqref{Vrdiffestimate}. We proceed analogously, and fix partitions $s = t_0^N \le t_1^N \le \dotsb \le t_N^N = t$ as in \eqref{e:equipartition0} (with $r'-\eps$ replaced by $r_1$), so that for every $x, y \in [t_j^N, t_{j+1}^N]$,
\begin{equation}\label{e:equipartition3}
\begin{gathered}
|f(x) - f(y)|^{r_1} \le \frac 2 N \|f\|_{V^{r_1}([s,t])}^{r_1}, \\
|(f-g)(x) - (f-g)(y)|^{r_1} \le \frac 2 N \|f-g\|_{V^{r_1}([s,t])}^{r_1}.
\end{gathered}
\end{equation}
In order to realise such a partition, we can construct $s = t_0^N(f) < t_1^N(f) < \dotsb < t_N^N(f) = t$ and $s = t_0^N(f-g) < t_1^N(f-g) < \dotsb < t_N^N(f-g) = t$ as in \eqref{e:equipartition0}, and then take $\{ t_0^{2N}, \dotsc, t_{2N}^N \} = \{ t_j^N(f) \}_{j \le N} \cup \{ t_j^N(f-g) \}_{j \le N}$, appropriately rearranged so they are in increasing order.

\noindent
For any function $u$ such that its Fourier transform has support
in $[-10N, 10N]$, we have the inequality 
\[
  \begin{aligned}
    \lefteqn{\|(\tau_{y_2} - \tau_{x_2})u - (\tau_{y_1} - \tau_{x_1})u\|_{L^p}}\qquad&\\
       &\le\|(\tau_{y_2} - \tau_{x_2})u - (\tau_{y_2 - (x_2-x_1)} - \tau_{x_2 - (x_2 - x_1)})u \|_{L^p}
         + \| (\tau_{y_2 - (x_2-x_1)} - \tau_{y_1}) u \| \\
       &\les N^2 \|u\|_{L^p}\,|x_2 - x_1|\, |y_2 - x_2| 
         + N \| u \|_{L^p}|(y_2 - x_2) - (y_1 - x_1)|,
  \end{aligned}
\]
and similarly, 
\[
  \| (\tau_{y_2} - \tau_{x_2}) u - (\tau_{y_1} - \tau_{x_1}) u \|_{L^p} 
    \les N \|u\|_{L^p}\bigl(|y_2 - x_2|  + |(y_2 - x_2) - (y_1 - x_1)|\bigr).
\]
Interpolating between the two, we obtain 
\[
  \begin{multlined}[.9\linewidth]
    \| (\tau_{y_2} - \tau_{x_2}) u - (\tau_{y_1} - \tau_{x_1}) u \|_{L^p} \\
      \les |x_2 - x_1|^\gamma |y_2 - x_2| N^{1+ \gamma} \|u\|_{L^p}
        + |(y_2 - x_2) - (y_1 - x_1)| N \| u \|_{L^p}.
  \end{multlined}
\]
Therefore, for any $x_1, y_1, x_2, y_2 \in \R^d$, we have
\begin{equation} \label{e:holdertranslation2}
  \begin{multlined}[.9\linewidth]
    \|(\tau_{y_2} - \tau_{x_2})u - (\tau_{y_1} - \tau_{x_1})u\|_{B^{\alpha - \gamma - 1}_{p,q}}\\
      \les \big(|x_2 - x_1|^\gamma |y_2 - x_2| + |(y_2 - x_2) - (y_1 - x_1)|\big) \| u \|_{B^{\alpha}_{p,q}}.
  \end{multlined}
\end{equation}
Notice that \eqref{e:gammarvar} still holds for $g$ and $f-g$, due to the choice of the partition \eqref{e:equipartition3}. Therefore, by \eqref{e:holdertranslation2} and \eqref{e:gammarvar}, we obtain that 
\[
  \begin{aligned}
    \lefteqn{\|(\mu^{f,2N}_{s,t} - \mu^{g,2N}_{s,t})
        - (\mu^{f,N}_{s,t} - \mu^{g,2N}_{s,t})\|_{B^{\alpha-\gamma - 1}_{p,q}}}\qquad&\\
      &=\Big\| \sum_{j=0}^{N-1} \big((\tau_{f(t_{2j+1}^{2N})} - \tau_{g(t_{2j+1}^{2N})}
        - (\tau_{f(t_{2j}^{2N})} - \tau_{g(t_{2j}^{2N})})\big)
        \mu_{t_{2j+1}^{2N}, t_{2(j+1)}^{2N}} \Big\|_{B^{\alpha-\gamma-1}_{p,q}}\\
      &\leq\sum_{j=0}^{N-1} \| \mu_{t_{2j+1}^{2N}, t_{2(j+1)}^{2N}}\|_{B^\alpha_{p,q}}\cdot\\
      &\qquad \cdot\Big(|g(t_{2j+1}^{2N}) - g(t_{2j}^{2N})|^\gamma \|f -g \|_{L^\infty([s,t])}
        + |(f-g)(t_{2j+1}^{2N}) - (f-g)(t_{2j}^{2N})|\Big)\\
      &\leq\Big(\sum_{j=0}^{N-1}
        \|\mu_{t_{2j+1}^{2N}, t_{2(j+1)}^{2N}}\|_{B^\alpha_{p,q}}^r \Big)^\frac1r
        \Big[\Big(\sum_{j=0}^{N-1}
          \big|g(t_{2j+1}^{2N}) - g(t_{2j}^{2N})\big|^{\gamma r'}\Big)^\frac 1{r'}
        \| f -g \|_{L^\infty([s,t])} + {}\\
      &\qquad  + \Big(\sum_{j=0}^{N-1}\big|(f-g)(t_{2j+1}^{2N})
          - (f-g)(t_{2j}^{2N})\big|^{r'}\Big)^\frac 1{r'}\Big]\\
      &\leq \|\mu \|_{V^r([s,t],B^\alpha_{p,q})}
        \Big(N^{-\frac{\gamma r' - r_1}{r'r_1}} \|g\|_{V^{r_1}}^\gamma \| f - g\|_{L^\infty}
        + N^{-\frac{r' - r_1}{r'r_1}} \|f-g\|_{V^{r_1}}\Big)
  \end{aligned}
\]
By telescopic summation, we obtain that 
\[
  \begin{aligned}
    \| \mu_{s,t}^f - \mu_{s,t}^g \|_{B^{\alpha - \gamma - 1}_{p,q}}
      &\le \| f - g \|_{L^\infty} \| \mu_{s,t}\|_{B^{\alpha - \gamma - 1}_{p,q}}\\
      &\quad + \|\mu \|_{V^r([s,t],B^\alpha_{p,q})}
        \Big( \|g\|_{V^{r_1}}^\gamma \| f - g\|_{L^\infty} + \|f-g\|_{V^{r_1}}\Big),
  \end{aligned}
\]
from which \eqref{Vrdiffestimate} follows easily.
\end{proof}
\section{\texorpdfstring{\sbe}{SBE} regularity for stochastic processes}\label{s:sbestochproc}

We wish to discuss here $\sbe$ regularity for paths of stochastic processes.
It turns out that it is sufficient to control the density of the increments
of paths to get \sbe regularity.
\begin{theorem} \label{thm:spSBE}
  Let $\beta >0$ and let $\omega:[a,b]\to\R^d$ be a stochastic process
  with continuous paths. For $a \le s,t \le b$, let 
  $\nu_{st} = \Law(\omega_{st})$, and suppose that for every
  interval $J \subseteq [a,b]$,
  \begin{equation} \label{eqn:Cnu}
    \iint_{J \times J} \| \nu_{st} \|_{C^\beta}\,ds\,dt
      \leq C_{\nu}|J|.
  \end{equation} 
  Then for every $\alpha < \frac\beta2$ with
  $2 \alpha \not \in \N$, the occupation measure satisfies
  \[
    \mu_{a,\cdot} \in V^2(\sbe^{\alpha,2}_{2}).
  \]
  More precisely, for every $\delta\in(0,1)$,
  and for every $1 \le q < 2$, we have that 
  \begin{equation}\label{e:moment}
    \E \Big[{\| \mu_{a,\cdot} \|_{V^2(\sbe^{\alpha, 2}_{2})}^q}
        (1 + \sup_{s} | \omega_{s} |)^{-d}\Big]
      \les 1 + C_\nu^{(1-\delta)\frac q2}.
  \end{equation}
\end{theorem}
The proof of the theorem can be found at page \pageref{pf:spSBE}.
A straightforward application of Kolmogorov's continuity theorem
yields the following result (see page \pageref{pf:holder} for the proof).
\begin{corollary}\label{c:holder}
  Given a path $\omega:[a,b]\to\R^d$, under the same assumptions
  of \autoref{thm:spSBE}, the occupation measure $\mu_{a,}$
  of $\omega$ is $\gamma$-H\"older
  continuous with values in $\sbe^{\alpha,2}_2$, for
  every $\gamma<\frac{\beta-2\alpha}{2(\beta+2d)}$.
\end{corollary}
\begin{remark}\label{r:preextend}
  Notice that the H\"older exponent in the above corollary
  is, no matter what is the value of parameters, smaller
  than $\frac12$. This apparently prevents the use of the
  Young integration theory developed in \cite{CatGub2016}.
  We will overcome the issue in two ways. On the one hand
  we shall directly develop Young's integration theory
  in spaces of finite variation in \autoref{s:young}.
  On the other hand the H\"older exponent can be
  improved in two ways. First, \eqref{eqn:Cnu} is apparently
  a bit restrictive and we could assume instead, without
  changing significantly our proofs, that
  \[
    \iint_{J\times J}\|\nu_{st}\|_{C^\beta}
      \leq C_\nu|J|^{2\eta},
  \]
  with $\eta\in[0,1]$. Under this assumption,
  \autoref{c:holder} yields $\gamma$-H\"older
  continuity with values in $\sbe^{\alpha,2}_2$
  for $\gamma<\frac12 - (1-\eta)(1-\delta_0)$,
  and $\delta_0$ is as in \eqref{e:deltazero}.
  Unfortunately, the H\"older exponent is
  still too small for our purposes.
  A second possibility is a trade-off
  between time and space regularity.
  Unfortunately this will require
  much stronger assumptions and will
  not follow from the slick condition
  \eqref{eqn:Cnu} of \autoref{thm:spSBE}.
  We postpone to \autoref{s:catgub}
  this discussion.
\end{remark}
The condition of \autoref{thm:spSBE}
simplifies dramatically when dealing with a Gaussian
process.
\begin{corollary} \label{cor:fbHSBE}  
Let $0 < H < 1$, let $\omega_t : [a,b] \to\R^d$ be a Gaussian process
with continuous paths. Suppose that for $a \le s,t \le b$,
\begin{equation}\label{e:simpleLND}
  \Cov(\omega_{st}) \gtrsim |t-s|^{2H}\mathrm{Id}.
\end{equation}
Then for every $\alpha < \frac 1 {2H}-\frac d2$,
\[
  \mu_{a,\cdot}^\omega
    \in V^2(\sbe^{\alpha,2}_{2}).
\]
\end{corollary}
\begin{proof}
  It is enough to check that $\omega$ satisfies the assumptions
  of \autoref{thm:spSBE} for any $\beta < \frac 1 H - d$.
  Since $\omega$ is Gaussian, we have that 
  \[
    \nu_{st}(x)
      = \frac{1}{(2 \pi)^\frac d2 \sqrt{\det(\Cov(\omega_{st}))}}
        \exp\Big(-\frac12(x-\E[\omega_{st}])^T\Cov(\omega_{st})^{-1}(x-\E[\omega_{st}])\Big),
  \]
  therefore, we have that 
  \[
    \|\nu_{st}\|_{C^\beta}
      \les |t-s|^{-H(\beta+d)}.
  \]
  Let $J=[a',b']\subseteq [a,b]$ be an interval.
  For $\beta < \frac 1 H - d$, we have 
  \[
    \iint_{J \times J} \|\nu_{st}\|_{C^{\beta}} ds dt
      \les  \iint_{J \times J} |t-s|^{-(\beta +d)H} ds dt \\
      \les \int_{2a'}^{2b'} \int_{-|J|}^{|J|} |s'|^{-(\beta +d)H}  ds' dt' \\
      \les |J|.\qedhere
  \]
\end{proof}
\begin{remark}
  The condition above in \autoref{cor:fbHSBE} can be
  reinterpreted as a property of local non-determinism.
  Indeed, if $\mathscr{F}_t=\sigma(\omega_s:a\leq s\leq t)$,
  for $t\in[a,b]$, is the filtration generated by the process
  $(\omega_t)_{t\in[a,b]}$, and if the assumption
  \[
    \Cov(\omega_t\mid\mathscr{F}_s)
      \gtrsim |t-s|^{2H}\mathrm{Id},
  \]
  then \eqref{e:simpleLND} holds. More on this will
  be considered in \autoref{s:catgub}, see
  \autoref{r:lnd}.
\end{remark}
\subsection{Proof of \autoref{thm:spSBE}}

We start with some preliminary results.
\begin{lemma} \label{lem:deltakmove}
  Let $k \ge 0$. We have for $x,y>0$ that
  \[
    \Delta_k \1_{[y,\infty)}(x)
      = \Delta_k^\ast \1_{[0,x]}(y).
  \]
\end{lemma}
\begin{proof}
For $k= 0$, we have that 
\begin{multline*}
  \Delta_0\1_{[y,\infty)}(x)
    = \1_{[y,\infty)}(x) - \1_{[y,\infty)}(x/2)\\
    = \1_{[0,x]}(y) - \1_{[0,x/2]}(y)
    = \1_{[0,x]}(y) - \1_{[0,x]}(2y)
    = \Delta_0^\ast \1_{[0,x]}(y).
\end{multline*}
Therefore, by proceeding inductively, for $k\ge1$, 
\begin{align*}
  \Delta_k\1_{[y,\infty)}(x)
    &= \Delta_{k-1}\1_{[y,\infty)}(x) -2^k \Delta_{k-1}\1_{[y,\infty)}(x/2)\\
    &= \Delta_{k-1}^\ast \1_{[0,x]}(y) - 2^k \Delta_{k-1}^\ast \1_{[0,x/2]}(y)\\
    &= \Delta_{k-1}^\ast \1_{[0,x]}(y) - 2^k \Delta_{k-1}^\ast \1_{[0,x]}(2y)\\
    &= \Delta_{k}^\ast \1_{[0,x]}(y).\qedhere
\end{align*}
\end{proof}

\begin{lemma} \label{lem:singledeltak}
Fix $\beta \ge 0$, $\alpha_0 \ge 0$, and let $k = \lceil \beta -1 \rceil$. Suppose that for some $R < \infty$, $f: \R \to \R$ satisfies $\| f \|_{C^{\beta+\alpha_0-k}(B_R)} < \infty$. Then 
\[
\| \Delta_k f \|_{C^{\alpha_0}(B_R)} \lesssim R^\beta \| f^{(k)}\|_{C^{\beta+\alpha_0-k}(B_R)}. 
\]
\end{lemma}
\begin{proof}
By \eqref{deltakernel}, we have that  $\Delta_k x^{h} = 0 $ as long as $h \le k$. Consider the function 
\[
  g(x):= f(x) - \sum_{h \le k} \frac{f^{(h)}(0)}{h!} x^{h}.
\]
We then have that 
\[
  \Delta_k g =  \Delta_k f.
\]
Moreover, by Taylor's remainder theorem, 
\[
  \|g\|_{C^\alpha_0(B_R)} \les R^\beta \| f^{(k)} \|_{C^{\beta+\alpha_0-k}}.
\]
By the definition of $\Delta_k$, we then have that 
\[
\|\Delta_k g\|_{C^{\alpha_0}(B_R)} \les \|g\|_{C^{\alpha_0}(B_R)} \les R^\beta \| f^{(k)} \|_{C^{\beta+\alpha_0-k}}.\qedhere
\]
\end{proof}

\begin{corollary}\label{cor:multideltak}
Let $m \in \N$, $\beta >0$, $\beta \not \in \N$, and let
$k = \lfloor \beta \rfloor$. Suppose that $f: \R^m \to \R$ satisfies
\[
  \| |\partial_{x_1}|^\beta \dots |\partial_{x_m}|^\beta f\|_{L^\infty}
    < \infty,
\]
where $|\partial_{x}|$ is defined as $|\partial_{x}| = \sqrt{-\partial_x^2}$.
Then for every $R > 0$, 
\begin{equation}\label{e:multideltak}
  \| (\Delta_k)_{x_1} \dots (\Delta_k)_{x_m} f\|_{L^\infty(B_R^m)}
    \les R^{m\beta}\|
      |\partial_{x_1}|^\beta\dots|\partial_{x_m}|^\beta f\|_{L^\infty}.
\end{equation}
Moreover, if $f \in C^{m\beta}(B_R^m)$, 
\begin{equation}\label{e:multideltakweak}
\| (\Delta_k)_{x_1} \dots (\Delta_k)_{x_m} f\|_{L^\infty(B_R^m)} \les R^{m\beta} \| f \|_{C^{m\beta}(B_R^m)}
\end{equation}
\end{corollary}
\begin{proof}
We proceed inductively. For $m=1$, \eqref{e:multideltak} follows directly from Lemma \ref{lem:singledeltak}, 
together with the usual equivalence (for $\beta \not \in \N$) of $C^\beta$ and $B^{\beta}_{\infty,\infty}$. If $m > 1$, by induction, 
\[
  \begin{aligned}
    \|(\Delta_k)_{x_1} \dots (\Delta_k)_{x_m} f\|_{L^\infty(B_R^m)}
      &\les R^{(m-1)\beta}\| |\partial_{x_1}|^\beta \dots
        |\partial_{x_{m-1}}|^\beta (\Delta_k)_{x_m}  f\|_{L^\infty(\R^{m-1} \times B_R)} \\
      &= R^{(m-1)\beta} \| (\Delta_k)_{x_m} |\partial_{x_1}|^\beta \dots
        |\partial_{x_{m-1}}|^\beta   f\|_{L^\infty(\R^{m-1} \times B_R)} \\
      &\les R^{m\beta} \||\partial_{x_{m}}|^\beta |\partial_{x_1}|^\beta \dots
        |\partial_{x_{m-1}}|^\beta   f\|_{L^\infty(\R^m)} \\
      &= R^{m\beta}\| |\partial_{x_1}|^\beta \dots |\partial_{x_m}|^\beta f\|_{L^\infty},
  \end{aligned}
\]
which is \eqref{e:multideltak}. To obtain \eqref{e:multideltakweak}, we proceed similarly. 
If $m = 1$, \eqref{e:multideltakweak} follows directly from Lemma \ref{lem:singledeltak}. For $m > 1$, proceeding inductively and using Lemma \ref{lem:singledeltak}, for $k:= \lfloor \beta \rfloor$,
\[
  \begin{aligned}
  \lefteqn{\| (\Delta_k)_{x_1} \dots (\Delta_k)_{x_m} f\|_{L^\infty(B_R^m)}}\qquad&\\
    &\les R^\beta \| (\Delta_k)_{x_1} \dots (\Delta_k)_{x_{m-1}}\partial_{x_m}^kf
      \|_{L^\infty_{x_1,\dotsc,x_{m-1}}(B_R^{m-1}) C^{\beta-k}_{x_m}(B_R)}\\
    &\les R^\beta \Big(\| (\Delta_k)_{x_1} \dots
      (\Delta_k)_{x_{m-1}}\partial_{x_m}^kf\|_{L^\infty(B_R^m)}+{} \\
    &\quad + \sup_{x_m, x_m' \in B_R} \Big\| \frac{(\Delta_k)_{x_1} \dots
      (\Delta_k)_{x_{m-1}}[(\partial_{x_m}^k f)(\cdot, x_m) - (\partial_{x_m}^k f)(\cdot,x_m')]}
      {|x_m-x_m'|^{\beta - k}} \Big\|_{L^\infty_{x_1,\dotsc,x_{m-1}}(B_R^{m-1})}\Big)\\
    &\les R^{m\beta} \Big(\|\partial_{x_m}^kf
      \|_{L^{\infty}_{x_m}(B_R)C^{(m-1)\beta}_{x_1,\dotsc,x_{m-1}}(B_R^{m-1})} + {}\\
    &\quad + \sup_{x_m, x_m' \in B_R} \Big\| \frac{(\partial_{x_m}^k f)(\cdot, x_m)
      - (\partial_{x_m}^k f)(\cdot, x_m')}{|x_m-x_m'|^{\beta - k}} \Big
      \|_{C^{(m-1)\beta}_{x_1,\dotsc,x_{m-1}}(B_R^{m-1})}\\
    &\les R^{m \beta} \| f \|_{C^{m\beta}(B_R^m)}.
  \end{aligned}\qedhere
\]
\end{proof}
\begin{lemma}
Let $F: \R \to \R$, $d \in \N$, $d \ge 1$. There exist coefficients $c_{h,d} \in \R$ such that for every $k \ge d$,
\begin{equation} \label{deltakleib}
\Delta_k (r^d F(r)) = \sum_{h=0}^{d} c_{h,d} r^d \Delta_{k-d} F\big(\frac{r}{2^h}\big).
\end{equation}
\end{lemma}
\begin{proof}
We proceed inductively. By expanding the definition \eqref{deltakdef} of $\Delta_k$, we have that 
\begin{equation*}
\Delta_{d-1} (r^d F(r)) = \sum_{h=0}^{d} c_{h,d}r^d F\big(\frac{r}{2^h}\big)
\end{equation*}
for some coefficients $c_{h,d}$. Therefore, by definition of $\Delta_k$,
\begin{align*}
\Delta_{d} (r^d F(r)) &= \sum_{h=0}^d c_{h,d} \big(r^d F\big(\frac{r}{2^h}\big) - 2^{d} \frac{r^d}{2^d} F\big(\frac{r}{2^{h+1}}\big) \big) \\
&= \sum_{h=0}^d c_{h,d} r^d \big( F\big(\frac{r}{2^h}\big) - F\big(\frac{r}{2^{h+1}}\big) \big) \\
&= \sum_{h=0}^{d} c_{h,d} r^d \Delta_{0} F\big(\frac{r}{2^h}\big),
\end{align*}
so \eqref{deltakleib} holds for $k = d$. Moreover, if \eqref{deltakleib} holds for some $k \ge d$, 
\begin{align*}
\Delta_{k+1} (r^d F(r)) &= \sum_{h=0}^d c_{h,d} \big(r^d \Delta_{k-d} F\big(\frac{r}{2^h}\big) - 2^{k+1} \frac{r^d}{2^d} \Delta_{k-d} F\big(\frac{r}{2^{h+1}}\big) \big) \\
&= \sum_{h=0}^d c_{h,d} r^d \big( \Delta_{k-d} F\big(\frac{r}{2^h}\big) - 2^{k+1-d}\Delta_{k-d} F\big(\frac{r}{2^{h+1}}\big) \big) \\
&= \sum_{h=0}^{d} c_{h,d} r^d \Delta_{k+1-d} F\big(\frac{r}{2^h}\big),
\end{align*}
which is \eqref{deltakleib} for $k+1$.
\end{proof}
We finally give a finite $p$-variation criterion for stochastic
processes under assumptions similar to those of Kolmogorov's
continuity theorem.
\begin{lemma}\label{l:magic}
  Let $E$ be a normed space, and
  let $f:[0,1]\to E$. Then for $q\geq 1$ and $\epsilon>0$,
  \[
    \|f\|_{V^q(0,1;E)}^q
      \lesssim \sum_{N\text{ dyadic}} \Bigl(
        N^\epsilon \sum_{k=0}^{N-1}\|f(t_{k+1}^N) - f(t_k^N)\|_E^q\Bigr),
  \]
  with a constant that depends only on $q$, $\epsilon$.
\end{lemma}
\begin{proof}
  Let $\pi=\{0=s_0<s_1<\dots<s_n=1\}$ be a partition of $[0,1]$, and for a dyadic number $N$ and an integer $0 \le k \le N$, let $t_k^N := \tfrac{k}{N}$. Define
  recursively the set of intervals $J_h^N$ by 
  \[
    \begin{cases}
      J_h^1
        = \{I = [t_k^1,t_{k+1}^1]: I\subseteq [s_h, s_{h+1}]\},\\
      J_h^N
        = \{I = [t_k^N,t_{k+1}^N]: I\subseteq [s_h, s_{h+1}],\ 
          I\not\subseteq J\ \forall J\in\cup_{M < N} J_h^M\},
    \end{cases}
  \]
  and let $J_h = \cup_{M} J_h^M$. 
  Informally, $J_h^N$ is the collection of all dyadic intervals of size
  $\sim N^{-1}$ that are not contained into any interval chosen previously
  (for a smaller value of $N$). Note that, by definition,
  $J_h^N$ has at most two elements: if not, there would be contiguous intervals
  that should have been chosen before. Therefore, by the
  H\"older inequality,
  \[
    \begin{aligned}
      \|f(s_{h+1}) - f(s_h)\|_E^q
        &= \Big\|\sum_{I\in J_h}f(\max I) - f(\min I)\Big\|_E^q\\
        &= \Big\|\sum_{N\text{ dyadic}}\sum_{I\in J_h^N}f(\max I) - f(\min I)\Big\|_E^q\\
        &\leq \Bigl(\sum_{N\text{ dyadic}} N^{-\frac\epsilon{q}}N^{\frac\epsilon{q}}
          \sum_{I\in J_h^N}\|f(\max I) - f(\min I)\|_E\Bigr)^q\\
        &\leq \Bigl(\sum_{N\text{ dyadic}} N^{-\frac{\epsilon q'}{q}}\Bigr)^{q-1}
          \Bigl(\sum_{N\text{ dyadic}} N^\epsilon
          \Bigl(\sum_{I\in J_h^N}\|f(\max I) - f(\min I)\|_E\Bigr)^q\Bigr)\\
        &\lesssim \sum_{N\text{ dyadic}} N^\epsilon
          \sum_{I\in J_h^N}\|f(\max I) - f(\min I)\|_E^q.
    \end{aligned}
  \]
  So, by summing up over the partition,
  \[
    \begin{aligned}
      \sum_{h=0}^{n-1}\|f(s_{h+1}) - f(s_h)\|_E^q
        &\lesssim \sum_{N\text{ dyadic}} N^\epsilon
          \sum_{h=0}^{n-1}\sum_{I\in J_h^N}\|f(\max I) - f(\min I)\|_E^q\\
        &\leq \sum_{N\text{ dyadic}} N^\epsilon
          \sum_{k=0}^{N-1}\|f(t_{k+1}^N) - f(t_k^N)\|_E^q.
    \end{aligned}
  \]
  Since the right-hand side is independent from the partition,
  the same inequality holds for the $p$-variation.
\end{proof}
\begin{proposition}\label{p:tolomeov}
  Let $(X_t)_{t\in[0,1]}$ be a stochastic process with values in
  a normed space $E$. Assume that given $M>0$ there are $p>1$, $a>0$
  and $C_M>0$ such that
  \[
    \E[\|X_t-X_s\|_E^p\1_{G(M)}]
      \leq C_M |t-s|^{1+a},
        \qquad s,t\in[0,1],
  \]
  for any event $G(M)\subset\{\sup_{s\in[0,1]}\|X_s\|_E\leq M\}$. Then for every $q$ such that
  $\frac{p}{1+a}<q\leq p$,
  \[
    \E\bigl[\|X\|_{V^q(0,1;E)}^p\1_{G(M)}\bigr]
      \lesssim C_M.
  \]
\end{proposition}
\begin{proof}
  By the assumption, for $\epsilon>0$,
  \begin{equation}\label{e:fromlemma}
    \begin{aligned}
      \E\Bigl[\sup_{N\text{\ dyadic}}\Bigl(N^{a-\epsilon}
          \sum_{k=0}^{N-1}\|X_{t^N_k t_{k+1}^N}\|_E^p\Bigr)\1_{G(M)}\Bigr]
        &\leq \E\Bigl[\1_{G(M)}\sum_{N\text{\ dyadic}}N^{a-\epsilon}
          \sum_{k=0}^{N-1}\|X_{t^N_k t_{k+1}^N}\|_E^p\Bigr]\\
        &\leq C_M\sum_{N\text{\ dyadic}}N^{a-\epsilon}
          \sum_{k=0}^{N-1} |t_{k+1}^N-t_k^N|^{1+a}\\
        &\leq C_M\sum_{N\text{\ dyadic}} N^{-\epsilon}
         \lesssim C_M
    \end{aligned}
  \end{equation}
  By \autoref{l:magic} and the H\"older inequality,
  \[
    \begin{aligned}
      \|X\|_{V^q}^q
        &\leq \sum_{N\text{ dyadic}} N^\epsilon
          \sum_{k=0}^{N-1}\|X_{t^N_k t_{k+1}^N}\|_E^q\\
        &\leq \sum_{N\text{ dyadic}} N^e
          \Bigl(N^{a-\epsilon}\sum_{k=0}^{N-1}\|X_{t^N_k t_{k+1}^N}\|_E^p\Bigr)^{\frac{q}{p}}\\
        &\lesssim \Bigl(\sup_{N\text{ dyadic}}N^{a-\epsilon}
          \sum_{k=0}^{N-1}\|X_{t^N_k t_{k+1}^N}\|_E^p\Bigr)^{\frac{q}{p}},
    \end{aligned}
  \]
  with $e=1 + (1+\tfrac{q}{p})\epsilon - \tfrac{q}{p}(1+a)<0$ by assumption,
  for $\epsilon$ small enough.
  By multiplying by $\1_{G(M)}$ the above inequality and
  by using \eqref{e:fromlemma}, the conclusion follows.
\end{proof}
We are ready to turn to the proof
of \autoref{thm:spSBE}. We start
with the following proposition,
which yields the same conclusions
under an assumption on the $2$-points
joint density of the path.
\begin{proposition} \label{prop:spSBE}
  Given $\beta>0$, let $\omega:[a,b]\to\R^d$ be a stochastic
  process with continuous paths. Suppose that
  \begin{itemize}
    \item for {a.\,e.} $(s,t)$, the joint density
      $\rho_{s,t}:=\Law(\omega_s,\omega_t)$ belongs
      to $C^{\beta}(B_R\times B_R)$ for every ball
      $B_R = \{ x\in\R^d:|x|\le R\}$,
    \item for every $R>0$ there is $C_\rho(R)>0$ such
      that for every $J\subseteq[a,b]$, 
      \begin{equation} \label{eqn:Crho}
        \iint_{J \times J} \|\rho_{s,t}\|_{C^\beta(B_R \times B_R)}\,ds\,dt
          \leq C_{\rho}(R)|J|.
      \end{equation} 
  \end{itemize}
  Then for every $\alpha<\frac\beta2$ with
  $2\alpha\not\in\N$, the occupation measure 
  satisfies
  \[
    \mu_{a,\cdot}
      \in V^2(\sbe^{\alpha,2}_{2}).
  \]
  More precisely, there is $\delta_0(\alpha,\beta)\in(0,1)$ such
  that for every $0<\delta<\delta_0(\alpha,\beta)$,
  if
  \[
    E(M)
      = \bigl\{ \sup_{a \le s \le b} |\omega_{s}|\le M \bigr\},
      \qquad
    F(M_0)
      = \bigl\{ \sup_{a \le s,t \le b} |\omega_{st}|\le M_0 \bigr\},
  \]
  then
  \begin{equation} \label{muinV2SBE}
    \E \big[ \1_{E(M)}\1_{F(M_0)}  \| \mu_{a,\cdot}\|_{V^2(\sbe^{\alpha,2}_{2})}^2 \big]
      \les M_0^{d\delta} (M^d C_\rho(2M))^{1- \delta} + M_0^d.
  \end{equation}
\end{proposition}
\begin{proof}
Notice that up to making $\beta$ smaller, we can
assume that $ 2 \lceil \alpha \rceil \ge \beta$, and that $\beta \not \in \N$.

Fix $s,t \in [a,b]$ and $M,M_0>0$. We start by estimating 
\[
  \E\bigl[\1_{E(M)}\1_{F(M_0)}\|\mu_{s,t}\|_{\sbe_2^{\alpha,2}}^2\bigr].
\]
Notice that
$F_{\mu_{s,t}}(\cdot,\cdot)=\int_s^t F_{\delta_{\omega_\tau}}(\cdot,\cdot)\,d\tau$.
Moreover, for $r \ge 0$, by \autoref{lem:deltakmove}
\[
  (\Delta_k)_r F_{\delta_{y_0}} (r,y)
    = (\Delta_k)_r \1_{[0,r]}(|y-y_0|)
    = (\Delta_k)_r \1_{[|y-y_0|, \infty)}(r)
    = (\Delta_k^\ast \1_{[0,r]})(|y-y_0|).
\]
Let $k = \lceil \alpha + d - 1 \rceil = \lceil \frac \beta 2 + d - 1 \rceil$,
and let $\chi \in C^\infty_c(\R)$ be such that $\chi(r) \equiv 1$
for $r \le 1$, $\supp(\chi) \subseteq [0,2]$,
and $\|\chi\|_{C^\beta} \les 1$. For $r \le 1$ and
$|y| \le M + 1$, we have that
\begin{align*}
  \lefteqn{\E \Big[ |\Delta_k F_{\mu_{s,t}}(\cdot,y)(r)|^2\Big]}\quad& \\
    &=\E\Big[\iint_{[s,t]^2} \Delta_k F_{\delta_{\omega_{s'}}}(\cdot,y)(r)
     \Delta_k F_{\delta_{\omega_{t'}}}(\cdot,y)(r)\,ds'\,dt'\Big]\\
    &=\iint_{[s,t]^2}  \E \big[(\Delta_{k}^\ast \1_{[0,r]})(|y-\omega_{s'}|)
      (\Delta_{k}^\ast \1_{[0,r]})(|y-\omega_{t'}|)\big]\,ds'\,dt' \\
    &=\iint_{[s,t]^2}\iint (\Delta_{k}^\ast \1_{[0,r]})(|y-y_1|)
      (\Delta_{k}^\ast \1_{[0,r]})(|y-y_2|) \rho_{s',t'}(y_1, y_2) dy_1dy_2 ds'dt'\\
    &=\iint_{[s,t]^2}\iint (\Delta_{k}^\ast \1_{[0,r]})(|y_1|)
      (\Delta_{k}^\ast \1_{[0,r]})(|y_2|) \rho_{s',t'}(y + y_1, y +y_2) dy_1dy_2 ds'dt',
    \intertext{and since $\Delta^*_k\1_{[0,r]}(|y|)=0$ if $|y|\geq1$,}
    &\begin{aligned}=\iint_{[s,t]^2}\iint (\Delta_{k}^\ast \1_{[0,r]})(|y_1|)
      (\Delta_{k}^\ast \1_{[0,r]})(|y_2|) \chi(|y_1|)\chi(|y_2|)\\
      \times \rho_{s',t'}(y + y_1, y +y_2) dy_1dy_2 ds'dt'
    \end{aligned}
\\
    &=\iint_{[s,t]^2}\iint (\Delta_{k}^\ast \1_{[0,r]})(r_1)
      (\Delta_{k}^\ast \1_{[0,r]})(r_2) r_1^d r_2^d
      g_{s',t'}(r_1,r_2)\frac{dr_1}{r_1}\frac{dr_2}{r_2}\,ds'\,dt',
\end{align*}
where
\[
  g_{s',t'}(r_1,r_2)
    := \iint_{(\Sb^{d-1})^2} \chi(r_1)\chi(r_2)
      \rho_{s',t'}(y + r_1\theta_1, y +r_2\theta_2)
      \,d\sigma(\theta_1) ,d\sigma(\theta_2).
\]      
      
We clearly have that $\|g_{s',t'}\|_{C^\beta}\les\|\rho_{s',t'}\|_{C^\beta(2M)}$.
Therefore, by \eqref{deltakleib} and \eqref{e:multideltakweak},
 \begin{align*}
    \lefteqn{\E[|\Delta_k F_{\mu_{s,t}}(\cdot,y)(r)|^2]}\quad& \\
     &= \iint_{[s,t]^2}\iint \1_{[0,r]}(r_1) \1_{[0,r]}(r_2)
       (\Delta_{k})_{r_1} (\Delta_{k})_{r_2} \big( r_1^d r_2^d g_{s',t'}\big)(r_1,r_2)
       \frac{dr_1}{r_1} \frac{dr_2}{r_2} ds'dt'  \\
    &\les \iint_{[s,t]^2} \iint_{\R^2} \1_{[0,r]}(r_1) \1_{[0,r]}(r_2)
      r_1^{d}r_2^{d}(r_1^{\beta} + 
        r_2^{\beta})
      \| g_{s',t'} \|_{C^\beta}  \frac{dr_1}{r_1} \frac{dr_2}{r_2} ds'dt'  \\
    &\les \iint_{[s,t]^2} r^{\beta + 2d} \| \rho_{s',t'} \|_{C^\beta(2M)}\,ds'\,dt'\\
    &\les C_\rho(2M) (r^{2d} \wedge r^{\beta+2d}) |t-s|, 
\end{align*}
from which we obtain, for $r\le1$,
\begin{equation}\label{spSBE1}
\E\Big[\1_{E(M)} \int |\Delta_k F_{\mu_{s,t}}(\cdot,y)(r)|^2 dy \Big] \les C_{\rho}(2M)M^d r^{\beta+2d} |t-s|.
\end{equation}
In the following, for simplicity of notation, we write $C_\rho = C_\rho(2M)$.
We have the simple estimate 
\begin{equation}\label{e:simple}
  \begin{aligned}
    \1_{F(M_0)}|\Delta_k F_{\mu_{s,t}}(\cdot,y)(r)|
      &\les \|F_{\mu_{s,t}}\|_{L^\infty} \1_{\{(|y-\omega_a| - M_0))_+
        \le r\le2^k(|y-\omega_a|+M_0) \}} \\
      & \les \|\mu_{s,t}\|_{TV}  \1_{\{(|y-\omega_a| - M_0))_+
        \le r\le2^k(|y-\omega_a|+M_0) \}} \\
      & = |t-s| \1_{\{(|y-\omega_a| - M_0))_+ \le r\le2^k(|y-\omega_a|+M_0) \}},
  \end{aligned}
\end{equation}
so
\begin{align}\label{spSBE3}
  \E \Big[\1_{F(M_0)} \int |\Delta_k F_{\mu_{s,t}}(\cdot,y)(r)|^2 dy \Big]
    \les (M_0^d+r^d)|t-s|^2. 
\end{align}
Therefore, from \eqref{spSBE1}, \eqref{spSBE3}, we have for $\delta\in[0,1)$,
\begin{equation}\label{spSBE5}
  \begin{multlined}[.9\linewidth]
    \E\bigl[\|\Delta_k  F_{\mu_{s,t}}(\cdot,y)(r))\|_{L^2_y}^2\1_{E(M)}\1_{F(M_0)}\bigr]\les\\
      \les (C_\rho M^d)^{1-\delta} M_0^{d\delta} r^{(\beta +2d)(1-\delta)} |t-s|^{1 + \delta} \1_{\{r\le1\}}+(M_0^d+r^d)|t-s|^2 \1_{\{r>1\}}.
  \end{multlined}
\end{equation}
Therefore, if $\delta < \delta_0$, and $\delta_0$ is such that
\begin{equation}\label{e:deltazero}
  (\beta + 2d)(1-\delta_0)
    = 2(\alpha+d),
\end{equation}
then
\begin{equation}\label{spSBE6}
  \E\bigl[\|\mu_{s,t}\|_{\sbe^{\alpha, 2}_{2}}^2\1_{E(M)}\1_{F(M_0)}\bigr]
    \les (C_\rho M^d)^{1-\delta} M_0^{\delta d} |t-s|^{1 + \delta} + M_0^d |t-s|^2.
\end{equation}
Inequality \eqref{muinV2SBE} now follows from  formula
above and \autoref{p:tolomeov}.
\end{proof}
\begin{proof}[Proof of \autoref{thm:spSBE}]\phantomsection\label{pf:spSBE}
We first consider the case in which, for the
joint density $\rho_{s,t}$ of $(\omega_s,\omega_t)$,
for every interval $J \subseteq [a,b]$, 
\begin{equation} \label{dnrhodecay}
\iint_{J \times J} \sup_{(x,y)\in \R^d\times\R^d} (1 + |x|^n + |y|^n) (\|d^{n} \rho_{s,t}(x,y) \|_\infty + \|\rho_{s,t}(x,y)\|_\infty)
   \le C'_{\rho} |J|,
\end{equation}
where $d^n$ is the total differential, and $n$ is an integer with $n > \beta + d$.
Let $X$ be a $d$-dimensional standard Gaussian random variable,
independent from $\omega$, and for $\lambda \in \R$,
consider the path $\omega^\lambda_{s} := \omega_s + \lambda X$,
and let $\mu^\lambda$ its associated occupation measure. We have that 
\begin{equation}\label{e:mulambda}
\mu^\lambda_{s,t} = \tau_{\lambda X} \mu_{s,t},
\end{equation}
so in particular, for every choice of the parameters $\alpha, p, q$,  
\begin{equation}\label{e:mulambdanorm}
\| \mu^\lambda_{s,t} \|_{\SBE} = \|\mu_{s,t}\|_{\SBE}.
\end{equation}
Moreover, since $X$ is independent from $\omega$, if $\rho^\lambda_{s,t} = \Law(\omega^\lambda_s, \omega^\lambda_t)$, we have 
\begin{align*}
\rho_{s,t}^\lambda(x,y) =  \frac{\lambda^{-d}}{ (2 \pi)^\frac d2} \int_{\R^d} \rho_{s,t}(x - z, y - z) \exp\big(-\frac{|z|^2}{2\lambda^2}\big) d z.
\end{align*}
From the identity 
\[
\nu_{s,t}(x-y) :=  \int_{\R^d} \rho_{s,t}(x - z, y - z) d z,
\]
and \eqref{dnrhodecay}, we obtain for $\lambda \ge 1$,
\[
  \begin{aligned}
     & \Big\|(2\pi\lambda^2)^\frac d2\rho_{s,t}^\lambda(x,y)-\exp\big(-\frac{|x|^2}{2\lambda^2}\big) \nu_{s,t}(x-y)\Big\|_{C^\beta((R^d)^2)}\\
      \les &~ \Big\| \int_{\R^d} \rho_{s,t}(x-z,y-z)\Big[\exp\big(-\frac{|z|^2}{2\lambda^2}\big)-\exp\big(-\frac{|x|^2}{2\lambda^2}\big)\Big]\,dz \Big\|_{C^\beta((R^d)^2)}\\
      \les &~  \Big\|d^n \int_{\R^d} \rho_{s,t}(x-z,y-z)\Big[\exp\big(-\frac{|z|^2}{2\lambda^2}\big)-\exp\big(-\frac{|x|^2}{2\lambda^2}\big)\Big]\,dz\Big\|_{L^\infty((R^d)^2)} \\
      &+ \Big\| \int_{\R^d} \rho_{s,t}(x-z,y-z)\Big[\exp\big(-\frac{|z|^2}{2\lambda^2}\big)-\exp\big(-\frac{|x|^2}{2\lambda^2}\big)\Big]\,dz \Big\|_{L^\infty((R^d)^2)}\\
      &\les \Big\|\Big(\frac{|x-z|}{\lambda} + \frac{1}{\lambda^{2}}\Big)(\|d^{n} \rho_{s,t}(x-z,y-z) \| + \|\rho_{s,t}(x-z,y-z)\|) \Big\|_{L^\infty((R^d)^2)}
\end{aligned}
\]
Therefore, for $J \subseteq [a,b]$, by \eqref{eqn:Cnu} and \eqref{dnrhodecay},
we have
\begin{equation}\label{spSBE8}
  \begin{aligned}
  &\iint_{J \times J} \| \rho_{s,t}^\lambda \|_{C^\beta(B_R \times B_R)}\,ds\,dt\\
    &\les \frac1{\lambda^d}\Big(\iint_{J \times J}
      \|(2 \pi \lambda^2)^\frac d2\rho_{s,t}^\lambda(x,y) - \exp\big(-\frac{|x|^2}{2\lambda^2}\big)\nu_{s,t}(x-y)
      \|_{C^\beta(B_R \times B_R)}\,ds\,dt\\
    &\quad + \iint_{J \times J} \|\exp\big(-\frac{|x|^2}{2\lambda^2}\big) \nu_{s,t}(x-y) \|_{C^\beta} \,ds\,dt\Big)\\
    &\les \lambda^{-d-1} C_{\rho}'|J| + \lambda^{-d} C_\nu|J|\\
    &\les \lambda^{-d} C_{\nu}|J| 
  \end{aligned}
\end{equation}
for $\lambda = \lambda(C_{\rho'}, C_\nu)$ big enough. 
Now, consider the events $E_\lambda(M) = \{ \sup_{a \le t \le b} |\omega^\lambda_{t}| \le M \}$
and  $F_\lambda(M_0) = \{ \sup_{a \le s,t \le b} |\omega^\lambda_{st}| \le M_0\}= F(M_0)$.
We have that 
\begin{align*}
  E_\lambda(M)
    &= \bigl\{\sup_{a\le s\le b} |\omega^\lambda_s - \lambda X| \le M \bigr\} \\
    &\subseteq \bigl\{\sup_{a\le s\le b} |\omega^\lambda_s|\le M + \lambda |X|\bigr\}
      \cap F_{\lambda}(2M)\\
    &\subseteq F_\lambda(2M) \cap
      \bigcup_{N\text{ dyadic}}\bigl\{\sup_{a\le s\le b} |\omega^\lambda_s|\le M + 4\lambda N \}
      \cap \{|X|\sim N \}\\
    &= F_\lambda(2M)\cap
      \bigcup_{N\text{ dyadic}} E_\lambda(M + 4 \lambda N)\cap\{|X|\sim N \}.
\end{align*}
Therefore, by applying \autoref{prop:spSBE} to the path $\omega^\lambda$,
and by \eqref{spSBE8}, we have that, for any $q < 2$, 
\begin{align*}
  \lefteqn{\E\bigl[\| \mu_{a,\cdot} \|_{V^2(\sbe^{\alpha, 2}_{2})}^q \1_{E(M)}\bigr]\leq}\qquad&\\
    &\leq \sum_{N\text{ dyadic}} \E\bigl[\|\mu^\lambda_{a,\cdot}\|_{V^2(\sbe^{\alpha, 2}_{2})}^q
      \1_{F_\lambda(2M)}\1_{E_\lambda(M + 4 \lambda N)}\1_{|X|\sim N}\bigr]\\
    &\leq \sum_{N \text{ dyadic}} \big(\E\bigl[
      \|\mu^\lambda_{a,\cdot}\|_{V^2(\sbe^{\alpha, 2}_{2})}^2
      \1_{F_\lambda(2M)}\1_{E_\lambda(M + 4 \lambda N)}\bigr]\big)^\frac q2
      \Prob[|X| \sim N]^\frac {2-q}2\\
    &\les \sum_{N \text{ dyadic}} \big((\lambda^{-d} C_\nu
      (M + 4 \lambda N)^d)^{1 - \delta} M^{d\delta} + M^d\big)^{\frac q 2}
      \exp\big(-\frac{2-q}2 cN^2\big),
\end{align*}
where we have used the simple fact that
$\|\mu_{a,\cdot}\|_{V^2(\sbe^{\alpha,2}_2)}=\|\mu^\lambda_{a,\cdot}\|_{V^2(\sbe^{\alpha,2}_2)}$.
By taking limits as $\lambda \to \infty$, we obtain that 
\[ 
  \E\bigl[\|\mu_{a,\cdot} \|_{V^2(\sbe^{\alpha, 2}_{2})}^q \1_{E(M)}\bigr]
    \le C_{q,\delta} C_\nu^\frac{(1-\delta)q} 2 M^{d\frac{\delta q}2} + M^{d\frac{ q}2} ,
\]
so in particular, we obtain
\begin{equation}\label{Cnuestimate}
\E \Big[{\| \mu_{a,\cdot} \|_{V^2(\sbe^{\alpha, 2}_{2})}^q}(1 + \sup_{s} | \omega_{s} |)^{-d}\Big] \les 1 + C_\nu^{(1-\delta)\frac q2}.
\end{equation}

We now move to the case in which $\rho_{s,t}$ does not
necessarily satisfy \eqref{dnrhodecay}. Let $\eps \ll 1$,
$M_1 \gg 1$ to be determined later. Let $Y_s :[a,b]\to \R$
be a process, independent from $\omega$, such that its density
satisfies \eqref{dnrhodecay},\footnote{For instance, one can
take a fractional Brownian motion with Hurst index $H \ll \frac1\beta$,
starting at $Y_0 \sim N(0,1)$.} and let $\sigma_{s,t}$ be the
joint density of $(Y_s, Y_t)$. Finally, let 
\begin{equation}\label{e:wepsilonMuno}
  \omega^{\eps, M_1}_{s} := H_{M_1}(\omega_s) + \eps Y_s,
\end{equation}
where $H_{M_1}$ is defined as 
\begin{equation}\label{e:HM1}
H_{M_1}(x) =
\begin{cases}
x & \text{ if } |x| \le M_1 \\
\frac{M_1x}{|x|} & \text{ if } x \ge M_1.
\end{cases}
\end{equation}
In particular, if $\rho_{s,t}^{M_1} = \Law(H_{M_1}(\omega_s), H_{M_1}(\omega_t))$, one has that 
\begin{equation} \label{spSBE10}
\rho_{s,t}^{M_1}(x,y) = \rho_{s,t}(x,y) \1_{|x|,|y| < M_1} + \widetilde{\rho}_{s,t}^{M_1}(x,y) = \rho_{s,t}(x,y) + \overline{\rho}_{s,t}^{M_1}(x,y),
\end{equation}
where $\widetilde{\rho}_{s,t}^{M_1}$ is a measure supported on 
\begin{equation} \label{spSBE11}
\supp(\widetilde{\rho}_{s,t}^{M_1}) \subseteq \{ |x| = M_1 \} \cup \{ |y| = M_1 \},
\end{equation}
and 
\[
  \|\widetilde{\rho}_{s,t}^{M_1} \|_{TV}
    = \Prob[\max (|\omega_s|, |\omega_t|) \ge M_1]
    \le \Prob[E(M_1/2)^c],
\]
and similarly,
\begin{equation} \label{spSBE13}
  \|\overline{\rho}_{s,t}^{M_1} \|_{TV}
    \leq 2\Prob[E(M_1/2)^c].
\end{equation}
By definition of $\omega^{\eps, M_1}_{s}$, recalling
that $Y_s$ is independent from the path $\omega$,
for $\rho_{s,t}^{\eps,M_1} = \Law(\omega^{\eps, M_1}_{s},
\omega^{\eps, M_1}_{t})$, by \eqref{spSBE10}, we have that
\begin{equation}  \label{spSBE14}
  \rho_{s,t}^{\eps,M_1}
    = \rho_{s,t}^{M_1} \star \eps^{-2}\sigma_{s,t}(\eps^{-1} \cdot)
    = \rho_{s,t} \star \eps^{-2}\sigma_{s,t}(\eps^{-1} \cdot)
      + \overline{\rho}_{s,t}^{M_1} \star \eps^{-2}\sigma_{s,t}(\eps^{-1} \cdot).
\end{equation}
Let 
\[
  \nu_{s,t}^{\eps,M_1}(x-y)
    := d\Law( \omega^{\eps, M_1}_{st})(x-y) = \int \rho_{s,t}^{\eps,M_1}(x-z,y-z) d z,
\]
and 
\[
  \nu^{\sigma,\eps}_{s,t}(x-y) :=  \eps^{-2} \int \sigma_{s,t}(\eps^{-1}(x-z),\eps^{-1}(y-z)) d z.
\]
From \eqref{spSBE14}, the properties of convolutions, and \eqref{spSBE13}, we have
\begin{align*}
\|\nu_{s,t}^{\eps,M_1} \|_{C_\beta} & \le \| \nu_{s,t} \star \eps^{-2}\sigma_{s,t}(\eps^{-1} \cdot) \|_{C^\beta} + \|\overline{\rho}_{s,t}^{M_1} \star \nu^{\sigma,\eps}_{s,t}\|_{C^\beta} \\
&\le \|\nu_{s,t}\|_{C^\beta} +  \|\overline{\rho}_{s,t}^{M_1}\|_{TV} \| \nu^{\sigma,\eps} \|_{C^\beta}\\
&\le  \|\nu_{s,t}\|_{C^\beta} + 2 \Prob[E(M_1/2)^c]\eps^{-\beta - d} \| \nu^{\sigma,1} \|_{C^\beta}
\end{align*}
Therefore, if we choose $M_1 = M_1(\eps)$ such that $\Prob[E(M_1/2)^c]\ll \eps^{\beta + d + 1}$, we obtain 
\[
\iint_{J\times J} \|\nu_{s,t}^{\eps,M_1} \|_{C_\beta} \le (C_\nu + \eps)|J|.
\]
By \eqref{spSBE10}, \eqref{spSBE11}, $\rho_{s,t}^{M_1}$ has compact support. Recalling that $\sigma_{s,t}$ satisfies \eqref{dnrhodecay}, then $\rho_{s,t}^{\eps,M_1}$ satisfies  \eqref{dnrhodecay} as well. Therefore, we can apply \eqref{Cnuestimate} to $\omega^{\eps,M_1}_{s}$, and we obtain
\[
  \E\Big[{\| \mu^{\eps,M_1(\eps)}_{a,\cdot} \|_{V^2(\sbe^{\alpha, 2}_{2})}^q}
      (1 + \sup_{s} | \omega^{\eps,M_1(\eps)}_{s} |)^{-d}\Big]
    \les 1 + (C_\nu + \eps)^{(1-\delta)\frac q2}.
\]
From the definition of $\Delta_k$ and $F_\mu$, it is not hard to check
that $\Delta_k F_\mu$ is a continuous function of $\omega \in C([a,b];\R^d)$,
see \autoref{l:nothard} below. As a consequence,
we have that $\|\mu\|_{V^2(\sbe^{\alpha, 2}_{2})}$ is a {l.\,s.\,c.} function
of $\omega \in C([a,b];\R^d)$. Noticing that $\omega^{\eps,M_1(\eps)} \to \omega$
in $C([a,b];\R^d)$ {a.\,s.} as $\eps \to 0$,
\[
  \begin{multlined}[.9\linewidth]
  \E \Big[{\| \mu_{a,\cdot} \|_{V^2(\sbe^{\alpha, 2}_{2})}^q}
      (1 + \sup_{s} | \omega_{s} |)^{-d}\Big]\leq\\
    \leq \liminf_{\eps \to 0} \E \Big[{\| \mu^{\eps,M_1(\eps)}_{a,\cdot}
      \|_{V^2(\sbe^{\alpha, 2}_{2})}^q}
      (1 + \sup_{s} | \omega^{\eps,M_1(\eps)}_{s} |)^{-d}\Big]\les\\
    \les C_\nu^{(1-\delta)\frac q2},
  \end{multlined}
\]
by Fatou's lemma.
\end{proof}
\begin{lemma}\label{l:nothard}
  Let $\omega^1,\omega^2:[a,b]\to\R^d$ be two continuous paths,
  and set $F_i(t,r,y)=F_{\mu_{a,t}^{\omega^i}}(r,y)$. Then for every
  $y\in\R^d$ and $1\leq p<\infty$,
  \[
    \bigl\|\sup_{t\in[a,b]}|F_1(t,\cdot,y)-F_2(t,\cdot,y)|\,\bigr\|_{L^p(\R^+)}
      \leq (b-a)\|\omega^1-\omega^2\|_\infty^{\frac1p}
  \] 
  In particular, the map $\omega\mapsto\mu^\omega$ is
  lower semi-continuous from $C([a,b])$ to $\sbe^{\alpha,p}_q$,
  for all $\alpha>0$, $1\leq p,q<\infty$.
  
  Finally, the map that, given $\omega\in C([a,b];\R^d)$,
  returns $t\mapsto\mu_{a,t}^\omega$
  is lower semi-continuous with values in $V^e(\sbe^{\alpha,p}_q)$,
  for all $e>0$.
\end{lemma}
\begin{proof}
  We have
  \[
    \begin{aligned}
    |F_1(t,r,y) - F_2(t,r,y)|
      &\leq \int_a^b |\1_{[0,r]}(|y-\omega_s^1|) - \1_{[0,r]}(|y-\omega_s^2|)|\,ds\\
      &= \int_a^b (\1_{\{|y-\omega^1_s|\leq r<|y-\omega^2_s|\}}
        + \1_{\{|y-\omega^2_s|\leq r<|y-\omega^1_s|\}})\,ds,
    \end{aligned}
  \]
  thus, by the H\"older inequality,
  \[
    \begin{aligned}
      \lefteqn{\int_0^\infty \|F_1(\cdot,r,y)-F_2(\cdot,r,y)\|_\infty^p\,dr\leq}\qquad&\\
        &\leq\int_0^\infty\Big|\int_a^b (\1_{\{|y-\omega^1_s|\leq r<|y-\omega^2_s|\}}
          + \1_{\{|y-\omega^2_s|\leq r<|y-\omega^1_s|\}})\,ds\Big|^p\,dr\\
        &\leq (b-a)^{p-1}\int_a^b\int_0^\infty(\1_{\{|y-\omega^1_s|\leq r<|y-\omega^2_s|\}}
          + \1_{\{|y-\omega^2_s|\leq r<|y-\omega^1_s|\}})\,dr\,ds\\
        &\leq (b-a)^{p-1}\int_a^b \big|\,|y-\omega_s^1| - |y-\omega_s^2|\,\big|\,ds\\
        &\leq (b-a)^p\|\omega^1-\omega^2\|_\infty.
    \end{aligned}
  \]
  This proves the first claim. The second and third claims follow
  by semi-continuity of the integral, and the fact that the estimate
  proved above is uniform in time.
\end{proof}
\begin{proof}[Proof of \autoref{c:holder}]\phantomsection\label{pf:holder}
  Under
  the assumptions of
  \autoref{prop:spSBE},
  \eqref{spSBE6} holds.
  Since $E(M)\subset F(2M)$, \eqref{spSBE6} reads, for every $M>0$
  as
  \[
    \E[\|\mu_{a,t}-\mu_{a,s}\|_{\sbe^{\alpha,2}_2}^2\1_{E(M)}]
      \leq C(M)|t-s|^{1+\delta},
  \]
  with $C(M)$ a number depending on $M$, and
  $\delta<\delta_0:=\frac{\beta-2\alpha}{\beta+2d}$,
  where $\delta_0$ is defined in \eqref{e:deltazero}.
  From here,
  it is not difficult to prove that $\mu_{a,\cdot}$
  is {a.\,s.} $\gamma$-H\"older continuous,
  for every $\gamma<\frac12\delta_0$ with
  values in $\sbe^{\alpha,2}_2$ on
  dyadic times. For instance, if
  \[
    U_\gamma
      = \sum_{n=0}^\infty\sum_{\substack{{s,t\text{ dyadic}}\\{0<|t-s|\leq 2^{-un}}}}
        \frac{\|\mu_{a,t}-\mu_{a,s}\|_{\sbe^{\alpha,2}_2}^2}{|t-s|^\gamma},
  \]
  then $\E[U_\gamma\1_{E(M)}]<\infty$, and the events $(E(M))_{M\geq1}$
  fill in the probability space. Then any standard proof of
  Kolmogorov's continuity theorem (for instance, see \cite{Dur2010})
  yields H\"older continuity on dyadic times with probability $1$.
  Finally, by \autoref{l:nothard},
  the map $t\mapsto\mu_{a,t}$ is lower semi-continuous in
  $\sbe^{\alpha,2}_2$ and thus, by density of dyadic times,
  it is H\"older continuous on all times. 
  In order to extend the result to the assumptions of \autoref{thm:spSBE}, 
  notice that from \eqref{e:mulambdanorm} and \autoref{prop:spSBE}
  applied to the path $\omega^\lambda$,
  we know that \eqref{spSBE6} holds for the occupation
  measure $\mu_{a,\cdot}^\lambda$ defined in \eqref{e:mulambda},
  if \eqref{dnrhodecay} holds. On the other hand,
  \eqref{dnrhodecay} holds for the density $\rho^{\epsilon,M_1}$
  of \eqref{e:wepsilonMuno}, thus \eqref{spSBE6} holds
  for $\mu_{a,\cdot}^{\epsilon,M_1}$. By \autoref{l:nothard},
  \eqref{spSBE6} holds for $\mu_{a,\cdot}$, and therefore
  the same proof given above yields H\"older regularity
  under the assumptions of \autoref{thm:spSBE}.
\end{proof}
\section{Space-time regularity trade-off}\label{s:catgub}

In view of \autoref{r:preextend} and the fact that there is
a gap between the regularity in variation given by \autoref{thm:spSBE}
and the lower H\"older regularity of \autoref{c:holder}, we discuss
how to fill in this gap, under a stronger assumption over regularity of joint laws
of increments.
\begin{theorem} \label{thm:spSBE-m}
  Given $\beta>0$, with $\beta \not \in \N$,
  and an integer $m\geq1$, let $\eta>0$
  be such that $\frac1{2m}\leq\eta\leq 1$.
  Let $\omega:[a,b]\to\R^d$ be a stochastic
  process with continuous paths, and denote by
  $\nu_{s_{1:2m}}$ the joint density of
  $(\omega_{s_2}-\omega_{s_1},\dots,\omega_{s_{2m}}-\omega_{s_{2m-1}})$,
  with $s_1,\dots,s_{2m}\in[a,b]$.

  Assume that there is $C_\nu>0$ such that for every
  interval $J\subseteq [a,b]$ and for every family
  $\beta_1,\dots, \beta_{2m-1}\in\R$ satisfying
  \begin{equation}\label{e:betas}
    \beta_j \in \{ 0, \beta , 2\beta \},
      \qquad
    \beta_1 + \dots + \beta_{2m-1} = 2m\beta,
  \end{equation}
  we have that
  \begin{equation} \label{eqn:Cnu-m}
    \int_{s_1 < s_2 < \dots < s_{2m},\, s_j \in J}\|(1-\Delta)_{x_1}^{\frac {\beta_1}2} \dotsb (1-\Delta)_{x_{2m-1}}^{\frac {\beta_{2m-1}}2} \nu_{s_{1:2m}}\|_{L^\infty}\,ds_{1:2m}
      \leq C_{\nu}|J|^{2m\eta}.
  \end{equation} 
  Then for every $\alpha<\beta$, with
  $\alpha\not\in\N$, the occupation measure satisfies
  \[
    \mu_{a,\cdot}^\omega
      \in V^{1/\eta}(\sbe^{\alpha,2m}_{2m}).
  \]
  More precisely, there is $\delta_0=\delta_0(\alpha,\beta)\in(0,1)$
  such that for every $\delta<\delta_0$,
  and for every $1\leq q<2m$,
  \[
    \E \Bigl[{\| \mu_{a,\cdot}^\omega \|_{V^{1/\eta}(\sbe^{\alpha,2m}_{2m})}^q}
        (1 + \sup_s|\omega_s|)^{-d}\Bigr]
      \les 1 + C_\nu^{(1-\delta)\frac{q}{2m}}.
  \]
\end{theorem}
These results immediately yield H\"older regularity
in time of the occupation measure.
\begin{corollary}
  Given a path $\omega:[a,b]\to\R^d$, under the same assumptions
  of \autoref{thm:spSBE-m}, the occupation measure $\mu_{a,}$
  of $\omega$ is $\gamma$-H\"older
  continuous with values in $\sbe^{\alpha,{2m}}_{2m}$, for
  every $\gamma<\eta+(1-\eta)\delta_0-\frac1{2m}$.
\end{corollary}
\begin{proof}
  The result follows by the same argument
  of \autoref{c:holder}, using \eqref{e:forholdernu}.
\end{proof}
We finally consider the case of Gaussian
processes.
\begin{corollary} \label{cor:fbHSBE-m}
  Let $0 < H < \frac {1}{d}$, and let $\omega_t : [a,b] \to\R^d$
  be a Gaussian process with continuous paths. Suppose moreover
  that $\omega$ is \emph{locally non-deterministic} with parameter
  $H$, that is, for every integer $n \ge 2$, there exist constants
  $c_n$ and $\delta_n$ such that for every sequence of times
  $a \le s_1 < s_2 < \dotsb < s_n \le b$, and every vector
  $(x_1, \dotsc, x_{n-1}) \in (\R^d)^{n-1}$,
  \begin{equation}\label{e:lnd}
    \Var\Big(\sum_{k=1}^{n-1} x_k \cdot (\omega_{s_ks_{k+1}}) \Big)
      \ge 2c_n \sum_{k=1}^{n-1} \|x_k\|^2 |s_{k+1}-s_k|^{2H}
  \end{equation}
  for every $s_1,\dotsc, s_n$ with $|s_n - s_1| \le \delta_n$.
  Let $\alpha \ge 0$, $1 \le p \le q \le \infty$ be such that 
  \begin{equation}\label{e:lndparams}
    \frac 1 p + \Big(\alpha-\frac dq\Big) H 
      < 1 - dH,
        \qquad
    \alpha 
      < \Bigl(\frac 1H - d\Bigr)\min\Bigl( \frac {1}{2},\frac {1}{q'}\Bigr),
  \end{equation}
  where $q'$ is the H\"older conjugate of $q$, {i.\,e.}
  $\frac 1{q'} = 1 - \frac 1 q$. Then we have that 
  \[
    \mu_{a,\cdot}^\omega \in C^\var{p}([a,b];B^{\alpha}_{q,\infty}).
  \]
\end{corollary}
\begin{remark}[Local non determinism]\label{r:lnd}
  Condition \eqref{e:lnd} says that the process is
  \emph{locally $\phi$-non deterministic}, with
  $\phi(r)=r^{2H}$, according to \cite{Cuz1978}.
  Local non determinism was first introduced
  by Berman in \cite{Ber1973}. A recent review
  on the subject can be found in \cite{Xia2006}.
  
  A notion of local non determinism for general
  processes (that is to say, non Gaussian),
  was given in \cite{Ber1983}. A process
  $\omega:[a,b]\to\R^d$ is \emph{locally
  $g$ non deterministic} if for every $n\geq2$
  there is $c_n0$ such that for all
  $s_1<s_2<\dots<s_n$,
  \[
    \nu_{s_{1:n}}(0,\dots,0)
      \geq c_n g(s_2-s_1)g(s_3-s_2)\dots g(s_n-s_{n-1}),
  \]
  where $\nu_{s_{1:n}}$ is the joint density
  of the increments $\omega_{s_2s_1},\dots,\omega_{s_{n-1}s_n}$.
\end{remark}
\begin{remark}
  Condition \eqref{e:betas} could be slightly weakened by adding an extra combinatorial
  condition to the sequence of $\beta_j$. Indeed, the requirement is that, given the sequence of
  $\beta_j$, then between each pair of values $2\beta$ there must be at least one
  $0$. Likewise, between each pair of values $0$ there must be at least
  a value $2\beta$. A short way to write the condition is that for
  all $1 \le n_1 \le n_2 \le 2m-1$,
  \[
    \sum_{j=n_1}^{n_2} \beta_j
      \leq (n_2-n_1+2)\beta.
  \]
  This condition though does not seem to be more useful or easier
  to check than \eqref{e:betas}.
\end{remark}
\subsection{Proof of \autoref{thm:spSBE-m} and \autoref{cor:fbHSBE-m}}

As in \autoref{prop:spSBE}, we first prove the final
result under assumptions on the joint density of the
path, instead of the joint density of increments.
\begin{proposition} \label{prop:spSBE-m}
  Given $\beta>0$, an integer $m\geq1$,
  $\frac 1{2m} \le \eta \le 1$,
  let $\omega:[a,b]\to\R^d$ be a stochastic
  process with continuous paths.
  For $(s_1,s_2,\dots,s_{2m}) \in [a,b]$,
  let $\rho_{s_{1:2m}}$ be the joint density
  of $(\omega_{s_1},\dots,\omega_{s_{2m}})$.
  Suppose that for every $J\subseteq[a,b]$,
  \begin{equation} \label{eqn:Crho-m}
    \int_{\substack{s_1<s_2<\dots<s_{2m}\\ s_j \in J}} \|(1-\Delta)^{\frac{\beta}2}_{x_1} \dots (1-\Delta)^{\frac {\beta}2}_{x_{2m}}\rho_{s_{1:2m}}\|_{L^\infty}\,ds_{1:2m}
          \leq C_{\rho}|J|^{2m \eta}.
      \end{equation} 
  Then for every $\alpha<\beta$, with
  $\alpha\not\in\N$, the occupation measure 
  satisfies
  \[
    \mu_{a,\cdot}
      \in V^{1/\eta}([a,b];\sbe^{\alpha,2m}_{2m}).
  \]
  More precisely, there is $\delta_0=\delta_0(\alpha,\beta)\in(0,1)$ such
  that for every $\delta<\delta_0$, 
  \begin{equation} \label{muinV2SBE-m}
    \E\big[\1_{E(M)}\1_{F(M_0)}\|\mu_{a,\cdot}\|_{V^{1/\eta}(\sbe^{\alpha,2m}_{2m})}^{2m} \big]
      \les M_0^{d\delta} (M^d C_\rho)^{1- \delta} + M_0^d.
  \end{equation}
\end{proposition}
\begin{proof}
  We follow the lines of the proof of \autoref{prop:spSBE}.
  It is not restrictive to assume that $\lceil\alpha\rceil\geq\beta$.
  First, for $s,t\in[a,b]$ and $M,M_0>0$, we estimate
  \[
    \E\bigl[\1_{E(M)}\1_{F(M_0)}\|\mu_{s,t}\|_{\sbe_{2m}^{\alpha,2m}}^{2m}\bigr].
  \]
  Let $k = \lceil\alpha+d-1\rceil=\lceil\beta+d-1\rceil$,
  and, for simplicity of notation, set 
  \begin{equation} \label{domaindef}
    D_{2m}(s,t)
      := \{ s_1 < \dots < s_{2m}: s_j \in [s,t]\text{ for all }1\le j\le 2m\}.  
  \end{equation}
  For $r\leq 1$ and $|y| \le M + 1$, by \autoref{lem:deltakmove},
  \[
    \begin{aligned}
      \lefteqn{\E \Big[ |\Delta_k F_{\mu_{s,t}}(\cdot,y)(r)|^{2m}\Big]}\quad& \\
    &=\E\Big[\int_{[s,t]^{2m}} \Delta_k F_{\delta_{\omega_{s_1}}}(\cdot,y)(r)
     \dots\Delta_k F_{\delta_{\omega_{s_{2m}}}}(\cdot,y)(r)\,ds_{1:2m}\Big]\\
    &=(2m)!\E\Big[\int_{D_{2m}(s,t)}\Delta_k F_{\delta_{\omega_{s_1}}}(\cdot,y)(r)
     \dots\Delta_k F_{\delta_{\omega_{s_{2m}}}}(\cdot,y)(r)\,ds_{1:2m}\Big]\\
    &\begin{multlined}
    =(2m)!\int_{D_{2m}(s,t)}\int_{(\R^d)^{2m}} (\Delta_{k}^\ast \1_{[0,r]})(|y-y_1|)
      \dots(\Delta_{k}^\ast \1_{[0,r]})(|y-y_{2m}|)\\
      \rho_{s_{1:2m}}(y_{1:2m})\,dy_{1:2m}\,ds_{1:2m}
    \end{multlined}
      \\
      &\begin{multlined}
      =(2m)!\int_{D_{2m}(s,t)}\int_{(\R^+)^{2m}} (\Delta_{k}^\ast \1_{[0,r]})(r_1)
      \dots(\Delta_{k}^\ast \1_{[0,r]})(r_{2m})\\
      \times r_1^d\dots r_{2m}^d
      g_{s_{1:2m}}(r_{1:2m})\frac{dr_{1:2m}}{r_1\dots r_{2m}}\,ds_{1:2m},
     \end{multlined}
   \end{aligned}
  \]
  where
  \[
    g_{s_{1:2m}}(r_{1:2m})
      := \int_{(\Sb^{d-1})^{2m}}
        \rho_{s_{1:2m}}(y + r_1\theta_1, \dots,y +r_{2m}\theta_{2m})
        \,d\sigma(\theta_{1:2m}).
  \]
We have that for every $\eps >0$, the operator $|\partial_{\theta}|^{\beta - \eps}  (1- \Delta)^{-\frac \beta 2}$ is bounded on $L^\infty(\R^d)$, with its norm independent from $\theta \in S^{d-1}$. 
Therefore, 
  \[
  \begin{aligned}
    \|\partial_{r_{1:2m}}^{\beta - \eps} g_{s_{1:2m}}\|_{L^\infty}
      &\le \Big\|\int_{(\Sb^{d-1})^{2m}} \prod_{j=1}^{2m} \partial_{\theta_{j}}^{\beta - \eps}
        \rho_{s_{1:2m}}(y + x_1, \dots,y +x_{2m}) \,d\sigma(\theta_{1:2m}) \Big\|_{L^\infty}\\
      &\les \|(1- \Delta)^{\frac \beta 2}_{x_{1:2m}}  \rho_{s_{1:2m}}\|_{L^\infty}.
  \end{aligned}
  \]
Therefore, by \eqref{deltakleib} and \eqref{e:multideltak},
  \[
    \begin{aligned}
      \lefteqn{\E[|\Delta_k F_{\mu_{s,t}}(\cdot,y)(r)|^{2m}]}\quad& \\
        &= (2m)!\int_{D_{2m}(s,t)}\int_{[0,r]^{2m}}
          (\Delta_{k})_{r_1}\dots(\Delta_{k})_{r_{2m}}
          \big(r_1^d\dots r_{2m}^d g_{s_{1:2m}}\big)(r_{1:2m})
          \frac{dr_{1:2m}}{r_1\dots r_{2m}}\,ds_{1:2m}\\
        &\les C_\rho r^{2m(\beta-\eps)+2md} |t-s|^{2m\eta},
    \end{aligned}
  \]
  and thus,
  \[
    \E\Bigl[\1_{E(M)} \int_{\R^d} |\Delta_k F_{\mu_{s,t}}(\cdot,y)(r)|^{2m}\,dy\Bigr]
      \les C_{\rho}M^d r^{2m(\beta-\eps)+2md} |t-s|^{2m\eta}.
  \]
  By interpolating the above estimate with the simple
  estimate \eqref{e:simple}, and by choosing $\eps$
  so that $\alpha + \eps < \beta$,
  \begin{equation}\label{e:forholder}
    \begin{multlined}[.9\linewidth]
      \E\bigl[\1_{E(M)}\1_{F(M_0)}\|\mu_{s,t}\|_{\sbe^{\alpha,2m}_{2m}}^{2m}\bigr]\les\\
        \les (C_\rho M^d)^{1-\delta} M_0^{d\delta} |t-s|^{2m(\eta(1-\delta) + \delta)}
          + M_0^d |t-s|^{2m},
    \end{multlined}
  \end{equation}
  for $\delta\in(0,\delta_0)$, with
  \begin{equation}\label{e:delta-m}
    (\beta+d)(1-\delta_0)
      = (\alpha+d).
  \end{equation}
  Inequality \eqref{muinV2SBE-m} follows from
  \autoref{p:tolomeov}, since
  $\eta(1-\delta) + \delta>\eta$.
\end{proof}
\begin{proof}[Proof of \autoref{thm:spSBE-m}]\phantomsection\label{pf:spSBE-m}
  For $J = [s,t]$, let $D_{2m}(J):=D_{2m}(s,t)$ be the
  set defined in \eqref{domaindef}.
  Assume first that for every interval $J\subset[a,b]$,
  \begin{equation}\label{e:dnrhodecay-m}
    \begin{multlined}
      \int_{D_{2m}(J)} \sup_{x_i\in\R^d}
          (1 + |x_1|^n + \dots +|x_{2m}|^n)\bigl(\|\rho_{s_{1:2m}}(x_{1:2m})\| + \|d^n
          \rho_{s_{1:2m}}(x_{1:2m})\|\bigr)\,ds_{1:2m}\\
        \leq C'_{\rho} |J|^{2m\eta},
    \end{multlined}
  \end{equation}
  where $\rho_{s_{1:2m}}$ is the density of
  $(\omega_{s_1},\dots,\omega_{s_{2m}})$,
  $d^n$ is the total differential,
  and $n$ is an integer with $n > 2m\beta + d$.

  For a $d$-dimensional standard Gaussian random variable $X$,
  independent from $\omega$, and for $\lambda \in \R$, define
  the path $\omega^\lambda_{s} := \omega_s + \lambda X$,
  and let $\mu^\lambda$ be its associated occupation measure, so
  that $\|\mu^\lambda_{s,t}\|_{\SBE}=\|\mu_{s,t}\|_{\SBE}$ for
  all $\alpha,p,q$. We have the identity
  \begin{equation}\label{e:identity}
    \begin{multlined}[.9\linewidth]
      \nu_{s_{1:2m}}(x_2-x_1,x_3-x_2,\dots,x_{2m}-x_{2m-1}) = {}\\
        = \int_{\R^d}\rho_{s_{1:2m}}(x_1-z,\dots,x_{2m}-z)\,dz.
    \end{multlined}
  \end{equation}
  By \eqref{e:identity} and \eqref{e:dnrhodecay-m}, for $\lambda\geq1$,
  by Sobolev embeddings we have 
  \begin{equation} \label{e:rhosplitm}
    \begin{aligned}
      \lefteqn{\Big\|\Dc_\beta
          \Big((2\pi\lambda^2)^\frac{d}2\rho_{s_{1:2m}}^\lambda(x_{1:2m})
          - e^{-\frac{|x_1|^2}{2\lambda}}
          \nu_{s_{1:2m}}(x_2-x_1,\dots,x_{2m}-x_{2m-1})\Big)\Big\|_{L^\infty(\R^{2md})}}\quad&\\   
        &= \Big\|\Dc_\beta\int_{\R^d}\big(
          e^{-\frac{|x_1|^2}{2\lambda^2}}
          - e^{-\frac{|z|^2}{2\lambda^2}}\big)
          \rho_{s_{1:2m}}(x_1 - z, \dotsc, x_{2m}-z)\,dz\Big\|_{L^\infty(\R^{2md})}\\
        &\les \Big\|d^n\int_{\R^d}\big(
          e^{-\frac{|x_1|^2}{2\lambda^2}}
          - e^{-\frac{|z|^2}{2\lambda^2}}
          \big)\rho_{s_{1:2m}}(x_1 - z, \dotsc, x_{2m}-z)\,dz\Big\|_{L^\infty(\R^{2md})}\\
        &\quad + \Big\|\int_{\R^d}\big(
          e^{-\frac{|x_1|^2}{2\lambda^2}}
          - e^{-\frac{|z|^2}{2\lambda^2}}
          \big)\rho_{s_{1:2m}}(x_1 - z, \dotsc, x_{2m}-z)dz\Big\|_{L^\infty(\R^{2md})}\\
        &\les \Big\|\bigl(\tfrac{|x_1-z|}{\lambda} + \tfrac{1}{\lambda^{2}}\bigr)
          (\|d^{n} \rho_{s,t}(x_1-z,\dotsc,x_{2m}-z) \| + |\rho_{s,t}(x_1-z,\dotsc,x_{2m}-z)\|)
          \Big\|_{L^\infty(\R^{2md})}
    \end{aligned}
  \end{equation}
  where $\Dc_\beta=(1-\Delta)^{\frac{\beta}2}_{x_1}\dots(1-\Delta)^{\frac {\beta}2}_{x_{2m}}$
  and $\rho_{s_{1:2m}}^\lambda$ is the density of
  $(\omega_{s_1}^\lambda,\dots,\omega_{s_{2m}}^\lambda)$.
  Moreover, we notice that for $\lambda \ge 1$, for every $\eps > 0$, by \eqref{eqn:Cnu-m}
  \begin{equation} \label{eqn:Cnu-m2}
    \begin{aligned}
      \lefteqn{\int_{D_{2m}(J)} \|\Dc_{\beta-\epsilon}\bigl(
          e^{-\frac{|x_1|^2}{2}}
          \nu_{s_{1:2m}}(x_1-x_2, \dotsc,x_{2m-1}-x_{2m})\bigr)
          \|_{L^\infty(\R^{2md})}\,ds_{1:2m}}\quad&\\
        &\les \int_{D_{2m}(J)} \|\Dc_{\beta-\epsilon}
          \nu_{s_{1:2m}}(x_1-x_2, \dotsc,x_{2m-1}-x_{2m})
          \|_{L^\infty(\R^{2md})}\,ds_{1:2m}\\
        &\begin{multlined}
        \les \int_{D_{2m}(J)}\Big\|
          \Big(\prod_{j=1}^{2m-1} \big((1-\Delta_{x_j})^{\frac \beta 2}
            + (1-\Delta_{x_{j+1}})^{\frac \beta 2}\big)
          \nu_{s_{1:2m}}\Big)\\
          (x_1-x_2, \dotsc,x_{2m-1}-x_{2m})
          \Big\|_{L^\infty(\R^{2md})}\,ds_{1:2m}
        \end{multlined}\\
        &\les C_\nu |J|^{2m\eta}.
  \end{aligned}
  \end{equation}
  Therefore, combining \eqref{e:rhosplitm} and \eqref{eqn:Cnu-m2}, we obtain
  \[
    \begin{aligned}
      \lefteqn{\int_{D_{2m}(J)} \|\Dc_{\beta-\epsilon}
          \rho_{s_{1:2m}}^\lambda(x_{1:2m})\|_{L^\infty(\R^{2md})}\,ds_{1:2m}}\quad&\\
        &\les \frac{1}{\lambda^d} \Bigl(\int_{D_{2m}(J)}
          \|(2\pi\lambda^2)^\frac d2 \rho_{s_{1:2m}}^\lambda(x_{1:2m})\\
        &\qquad - e^{-\frac{|x_1|^2}{2}}
          \nu_{s_{1:2m}}(x_2-x_1,\dots,x_{2m}-x_{2m-1})\|_{C^{2m\beta}(\R^{2md})} ds_{1:2m}\\
        &\quad +\int_{D_{2m}(J)} \big\|\Dc_{\beta-\epsilon}
          e^{-\frac{|x_1|^2}{2}}
          \nu_{s_{1:2m}}(x_2-x_1,\dots,x_{2m}-x_{2m-1})\|_{L^\infty(\R^{2md})} ds_{1:2m}\Bigr)\\
        &\les \lambda^{-d} C_{\nu}|J|^{2m\eta}
  \end{aligned}
  \]
  for $\lambda = \lambda(C_{\rho}', C_\nu)$ big enough. 

  Fixing $2 \eps < \beta - \alpha$, we can use the above estimate, \eqref{e:forholder} applied
  to $\omega^\lambda$, and the fact $\|\mu_{s,t}\|_{\sbe^{\alpha,2m}_{2m}}
  =\|\mu^\lambda_{s,t}\|_{\sbe^{\alpha,2m}_{2m}}$,
  to obtain that for every $1\leq q<2m$ and $\delta<\delta_0$
  (defined as in \eqref{e:delta-m}),
  \[
    \begin{aligned}
      \lefteqn{\E\bigl[\|\mu_{s,t}\|_{\sbe^{\alpha, 2m}_{2m}}^q\1_{E(M)}\bigr]\leq}\qquad&\\
    &\leq \sum_{N\text{ dyadic}} \E\bigl[\|\mu^\lambda_{s,t}\|_{\sbe^{\alpha,2m}_{2m}}^q
      \1_{F_\lambda(2M)}\1_{E_\lambda(M + 4 \lambda N)}\1_{|X|\sim N}\bigr]\\
    &\leq \sum_{N \text{ dyadic}} \big(\E\bigl[
      \|\mu^\lambda_{s,t}\|_{\sbe^{\alpha,2m}_{2m}}^{2m}
      \1_{F_\lambda(2M)}\1_{E_\lambda(M + 4 \lambda N)}\bigr]\big)^\frac q{2m}
      \Prob[|X| \sim N]^\frac {2m-q}{2m}\\
    &\les \sum_{N \text{ dyadic}}
      \bigl((\lambda^{-d} C_\nu(M + 4 \lambda N)^d)^{1 - \delta} M^{d\delta}
        + M^d\bigr)^{\frac{q}{2m}}
        |t-s|^{q(\delta+\eta(1-\delta))}
      e^{-\tfrac{cN^2}{2m}(2m-q)}
    \end{aligned}
  \]
  and by taking the limit $\lambda\to\infty$,
  \begin{equation}\label{e:forholdernu}
    \E\bigl[\|\mu_{s,t}\|_{\sbe^{\alpha, 2m}_{2m}}^q\1_{E(M)}\bigr]
      \lesssim \bigl(C_\nu^{(1-\delta)} M^{d\delta} + M^d\bigr)^{\frac{q}{2m}}
        |t-s|^{q(\delta+\eta(1-\delta))}.
  \end{equation}
  Likewise, by \eqref{muinV2SBE-m} for $\omega^\lambda$
  and the fact $\|\mu_{a,\cdot}\|_{V^p(\sbe^{\alpha,2m}_{2m})}
  =\|\mu^\lambda_{a,\cdot}\|_{V^p(\sbe^{\alpha,2m}_{2m})}$,
  we obtain for $1\leq q<2m$, $\delta<\delta_0$,
  and $p(\delta+\eta(1-\delta))>1$,
  \[
    \E\bigl[\|\mu_{a,\cdot} \|_{V^p(\sbe^{\alpha,2m}_{2m})}^q \1_{E(M)}\bigr]
      \lesssim\bigl(C_\nu^{(1-\delta)} M^{d\delta} + M^d\bigr)^{\frac{q}{2m}},
  \]
  and in particular,
  \begin{equation}\label{e:Cnuestimate-m}
    \E \Big[{\|\mu_{a,\cdot}\|_{V^p(\sbe^{\alpha,2m}_{2m})}^q}
        (1 + \sup_{s}|\omega_{s}|)^{-d}\Big]
      \les 1 + C_\nu^{(1-\delta)\frac{q}{2m}}.
  \end{equation}
  To get rid of assumption \eqref{e:dnrhodecay-m}, we proceed
  as in the proof of \autoref{thm:spSBE}. Define
  $\omega^{\eps, M_1} := H_{M_1}(\omega) + \eps Y$,
  where $H_{M_1}$ is given in \eqref{e:HM1}, and
  $Y$ is a process, independent from $\omega$,
  that satisfies \eqref{e:dnrhodecay-m}.
  Given $s_1,s_2,\dots,s_{2m}$, we can write
  (as in the proof of \autoref{thm:spSBE})
  $\rho_{s_{1:2m}}^{M_1}=\rho_{s_{1:2m}}+\overline\rho_{s_{1:2m}}^{M_1}$,
  where $\rho_{s_{1:2m}}^{M_1}$ is the density of
  $(H_{M_1}(\omega_{s_1}),\dots,H_{M_1}(\omega_{s_{2m}}))$
  and
  \begin{equation}\label{e:rhobarTV}
    \|\overline\rho_{s_{1:2m}}^{M_1}\|_{TV}
      \leq 2\Prob[E(M_1/2)^c].
  \end{equation}
  Let now $\rho^{\epsilon,M_1}_{s_{1:2m}}$ be the density
  of $(\omega_{s_1}^{\epsilon,M_1},\dots,\omega_{s_{2m}}^{\epsilon,M_1})$,
  let $\sigma_{s_{1:2m}}$ the density of
  $(Y_{s_1},\dots,Y_{s_{2m}})$ and
  $\sigma^\epsilon_{s_{1:2m}}=\frac1{\epsilon^{2m}}\sigma_{s_{1:2m}}(\cdot/\epsilon)$
  the density of $\epsilon(Y_{s_1},\dots,Y_{s_{2m}})$, then
  \begin{equation}\label{e:rhorhobar}
    \rho_{s_{1:2m}}^{\epsilon,M_1}
      = \rho_{s_{1:2m}}^{M_1}\star \sigma_{s_{1:2m}}^\epsilon
        + \overline\rho_{s_{1:2m}}^{M_1}\star\sigma_{s_{1:2m}}^\epsilon.
  \end{equation}
  We wish to apply the conclusions of the first part
  of this proof to $\omega^{\epsilon,M_1}$. First,
  since the joint density of $(Y_{s_1},\dots,Y_{s_{2m}})$
  satisfies \eqref{e:dnrhodecay-m}, the same is true
  for $\rho^{\epsilon,M_1}_{s_{1:2m}}$.
  It remains to verify that the joint density
  $\nu_{s_{1:2m}}^{\epsilon,M_1}$
  of $(\omega_{s_2}^{\epsilon,M_1}-\omega_{s_1}^{\epsilon,M_1},
  \dots,\linebreak \omega_{s_{2m}}^{\epsilon,M_1}-\omega_{s_{2m-1}}^{\epsilon,M_1})$
  satisfies \eqref{eqn:Cnu-m}. Let $\nu^{\sigma,\epsilon}_{s_{1:2m}}$
  be the joint density of
  $\epsilon(Y_{s_2}-Y_{s_1},\dots,\linebreak Y_{s_{2m}}-Y_{s_{2m-1}})$,
  then by \eqref{e:identity}, \eqref{e:rhobarTV}
  and \eqref{e:rhorhobar},
  \[
    \begin{aligned}
      \|\Dc\nu_{s_{1:2m}}^{\epsilon,M_1}\|_{L^\infty}
        &\leq \|(\Dc\nu_{s_{1:2m}})\star\sigma_{s_{1:2m}}^\epsilon\|_{L^\infty}
          + \|\overline{\rho}_{s_{1:2m}}^{M_1}*\nu^{\sigma,\epsilon}_{s_{1:2m}}\|_{C^\beta}\\
        &\leq \|\Dc\nu_{s_{1:2m}}\|_{L^\infty}
          + \|\overline{\rho}_{s_{1:2m}}^{M_1}\|_{TV}
            \|\nu^{\sigma,\epsilon}_{s_{1:2m}}\|_{C^\beta}\\
        &\les \|\Dc\nu_{s_{1:2m}}\|_{L^\infty}
          + \epsilon^{-(2m\beta+d)}\|\nu^{\sigma,1}_{s_{1:2m}}\|_{C^\beta}\Prob[E(M_1/2)^c],
    \end{aligned}
  \]
  with $\Dc=(1-\Delta_{x_1})^{\beta_1/2}\dots(1-\Delta_{x_{2m-1}})^{\beta_{2m-1}/2}$
  for brevity.
  If we choose $M_1=M_1(\epsilon)$ so that
  $\Prob[E(M_1/2)^c]\leq \epsilon^{2m\beta+d}C_\nu$,
  in conclusion $\nu^{\epsilon,M_1}_{s_{1:2m}}$
  satisfies \eqref{eqn:Cnu-m}, thus
  \eqref{e:Cnuestimate-m} yields,
  \[  
    \E\Bigl[{\| \mu^{\epsilon,M_1(\epsilon)}_{a,\cdot} \|_{V^p(\sbe^{\alpha,2m}_{2m})}^q}
      (1 + \sup_s | \omega^{\epsilon,M_1(\epsilon)}_s|)^{-d}\Bigr]
    \les 1 + (C_\nu)^{(1-\delta)\frac{q}{2m}},
  \]
  for $p(\delta+\eta(1-\delta))>1$, $1\leq q<2m$ and $\delta<\delta_0 $,
  as well as \eqref{e:forholdernu}.
  The conclusion now follows again
  by semi-continuity, thanks to \autoref{l:nothard}.
\end{proof}
\begingroup
\allowdisplaybreaks
\begin{proof}[Proof of \autoref{cor:fbHSBE-m}]
  First of all, we can assume without loss of generality
  that $\delta_n = |b-a|$. If this is not the case, we
  can apply the result with $\delta_n = |b-a|$ to the
  intervals $[a, a+\delta_n]$, $[a+\delta_n, a+2\delta_n]$,
  $[a+2\delta_n, a+3\delta_n]$ and so on, and obtain that
  $\mu_{a+ k\delta_n,a+(k+1)\delta_n} \in C^\var{q}B^{\alpha}_{p,\infty}$.
  By patching the intervals back together, we obtain that
  $\mu_{a,b} \in C^\var{q}B^{\alpha}_{p,\infty}$ as well.
  
  We first check the statement in the case $q=2m$, with $m\in \N$.
  By Sobolev embeddings and \eqref{besovestimate}, it is enough
  to prove that for every $\alpha < \frac 1{2H} - \frac d 2$
  and $p$ such that $ \frac 1p <1 -  dH  -\frac{dH}{2m}+\alpha H$,  
  \[
    \mu_{a,\cdot}^\omega
      \in C^\var{p}([a,b];\sbe^{\alpha,2m}_{2m}).
  \]
  By \autoref{thm:spSBE-m}, it is enough to check that
  \eqref{eqn:Cnu-m} holds for $\beta < \frac 1{2H} - \frac d 2$
  and $\eta = 1 - (d+\beta)H -\frac{dH}{2m}$.
  Indeed, if this holds, we can choose $\beta$ in such a way that 
  \[
    \alpha
      < \beta
      < \frac 1{2H} - \frac d 2,
        \qquad
    \frac 1p
      < \eta
      = 1 - (d+\beta)H + \frac{dH}{2m}.
  \]
  Recall that if a Gaussian random variable $X$ has
  covariance matrix $C$ and mean $\bar x$, then 
  \[
    \E[e^{-i\xi\cdot X}]
      = \exp\Big(-i\bar x \cdot \xi - \frac12 \langle C\xi,\xi\rangle\Big).
  \]
  For simplicity of notation, let
  $D_{2m}(J) := \{s_1 < s_2 < \dots < s_{2m}: \forall j, \, s_j \in J\}$.
  We have that  
    \begin{align*}
       \lefteqn{\int_{D_{2m}(J)}\|
           (1-\Delta_{x_1})^{\frac {\beta_1}2} \dotsb
           (1-\Delta_{x_{2m-1}})^{\frac {\beta_{2m-1}}2}
           \nu_{s_{1:2m}}\|_{L^\infty}\,ds_{1:2m}}\quad&\\
         &\le \int_{D_{2m}(J)} \Big\|
           \prod_{j=1}^{2m-1}(1 + |\xi_j|^2)^{\frac{\beta_j}{2}}
           \widehat{\nu_{s_{1:2m}}}(\xi_{1:2m-1})
           \Big\|_{L^1((\R^d)^{2m-1})}\,ds_{1:2m}\\
         &=\int_{D_{2m}(J)} \Big\|
           \prod_{j=1}^{2m-1} (1 + |\xi_j|^2)^{\frac{\beta_j}{2}}
           \exp\Big[-\frac12\Var\Big(\sum_{k=1}^{n-1}
           \xi_k \cdot (\omega_{s_ks_{k+1}})\Big)\Big]
           \Big\|_{L^1((\R^d)^{2m-1})}\,ds_{1:2m}\\
         &\le\int_{D_{2m}(J)}\Big\|
           \prod_{j=1}^{2m-1} (1 + |\xi_j|^2)^{\frac{\beta_j}{2}}
           \exp\Big[-c_n \sum_{k=1}^{n-1} \|\xi_k\|^2 |s_{k+1}-s_k|^{2H} \Big]
           \Big\|_{L^1((\R^d)^{2m-1})}\,ds_{1:2m}\\
         &\les \int_{D_{2m}(J)} \prod
           |s_{k+1}-s_k|^{-(d+\beta_j)H}\,ds_{1:2m}\\
         &\les |J| \prod_{j=1}^{2m-1} |J|^{1- (d+\beta_j)H} \\
         &\les |J|^{2m(1- dH - \beta H +\frac d{2m})}\\
         &\les |J|^{2m\eta},
    \end{align*}
  where we used the condition $\beta_j \le 2 \beta < \frac 1 H - d$
  in order to integrate over $D_{2m}(J)$, and the condition
  $\sum_j \beta_j= 2m\beta$ in order to evaluate the last product.
  
  We now move to the general case. By the usual interpolation
  inequalities between variation spaces and Besov spaces, we
  have that for every $1 \le p_1,q_1, p_2,q_2 \le \infty,\alpha_2$
  and for every $0 \le \theta \le 1$, by defining 
  \[
    \frac{1}{p_\theta}
      := \frac{\theta}{p_1} + \frac{1-\theta}{p_2},
        \qquad
    \frac{1}{q_\theta}
      := \frac{\theta}{q_1} + \frac{1-\theta}{q_2},
        \qquad
    \alpha_\theta
      := \theta\alpha_1+ (1-\theta)\alpha_2,
  \]
  we have that 
  \[
    \| \mu \|_{C^{p_\theta-\rm{var}}B^{\alpha_\theta}_{q_\theta,\infty}}
      \les \|\mu\|_{C^\var{p_1}B^{\alpha_1}_{q_1,\infty}}^\theta
        \|\mu\|_{C^\var{p_2}B^{\alpha_2}_{q_2,\infty}}^{1-\theta}
  \]
  Therefore, by interpolating between even integers, we obtain
  that for every $ 2 \le q < \infty$,
  $\mu_{s,t} \in C^\var{p}B^{\alpha}_{q,\infty}$
  for every $2 \le q < \infty$, $\alpha < \frac{1}{2}\big(\frac 1H - d\big)$,
  $\frac 1p < 1 - (d+\alpha)H + \frac{dH}{q}$. By Sobolev embeddings, we can
  extend the result to $q = \infty$. Therefore, it only remains
  to check the result for $1 \le q < 2$. We observe that in the case
  $p=q=1$, $\alpha = 0$, we have that 
  \[
    \|\mu_{s,t}^\omega\|_{B^{0}_{1,\infty}}
      \le \int_{s}^t \| \delta_{\omega_{\tau}} \|_{B^{0}_{1,\infty}}\,d\tau
      \les |t-s|,
  \]
  therefore $\mu_{a,b} \in C^\var{1}B^{0}_{1,\infty}$. Finally, the case
  where $1 \le q < 2$ follows by interpolation between the case $q=2$,
  $\alpha < \frac12 \big(\frac 1H - d\big)$,
  $\frac 1p < 1- \frac{dH}{2} - \alpha H$,
  and the case $p=q=1, \alpha = 0$.
\end{proof}
\endgroup
\section{Integration and ODEs with \texorpdfstring{\sbe}{SBE} paths}\label{s:young}

We briefly show how to solve ODEs driven by $\sbe$ paths.
We first recall a few notions and results from
\cite{FriVic2010}. Given a normed space $E$
and $T>0$, we denote by $C^\var{r}([0,T];E)$
the space of continuous paths of finite
$r$-variation with values on $E$. Likewise,
if $\Delta_T=\{(s,t):0\leq s\leq t\leq T\}$
we define $C^\var{r}(\Delta_T;E)$.
We also recall a sewing lemma for
finite variation paths (see for instance
\cite[Theorem 2.2]{FriZha2018}),
\begin{lemma}\label{l:sewing}
  Let $\chi:\Delta_2\to E$ be continuous and assume
  there are two controls $\rho,\sigma$ and
  two numbers $a,b>0$ with $a+b>1$
  such that
  \[
    \|\delta\chi_{sut}\|_E
      \leq \rho_{su}^a\sigma_{ut}^b,
      \qquad
      s\leq u\leq t,
  \]
  where $\delta\chi_{sut} = \chi_{st} - \chi_{su} - \chi_{ut}$.
  Then there is a unique $\Is\chi\in C([0,T];E)$ such that
  $\Is\chi_0=0$ and
  \begin{equation}\label{e:characterize}
    \|\Is\chi_{st} - \chi_{st}\|_E
      \leq \rho_{st}^a\sigma_{st}^b.
  \end{equation}
  Moreover, if $\chi\in C^\var{r}(\Delta_T;E)$, then
  $\Is\chi\in C^\var{r}(0,T;E)$.
\end{lemma}
We shall use the above lemma to define the integral
\[
  \int_0^T f(s,\theta_s-\omega_s)\,ds
\]
for suitable paths $(\theta_s)_{s\in[0,T]}$,
$(\omega_s)_{s\in[0,T]}$ on $\R^d$,
and non-smooth $f:\R^d\to\R^d$.
To this end we recall a standard result
for convolutions (see for instance
\cite{KhuSch2021}).
\begin{theorem}\label{t:convolution}
  Let $\alpha_1,\alpha_2\in\R$, $q,q_1,q_2\in(0,\infty]$,
  $p,p_1,p_2\in[1,\infty)$ be such that
  \[
    \frac1q
      \leq\frac1{q_1}+\frac1{q_2},
        \qquad
    1 + \frac1p
      = \frac1{p_1} + \frac1{p_2}.
  \]
  If $f\in B^{\alpha_1}_{p_1,q_1}$ and $g\in B^{\alpha_2}_{p_2,q_2}$,
  then $f\star g\in B^{\alpha_1+\alpha_2}_{p,q}$ and
  \[
    \|f\star g\|_{B^{\alpha_1+\alpha_2}_{p,q}}
      \lesssim \|f\|_{B^{\alpha_1}_{p_1,q_1}}
        \|g\|_{B^{\alpha_2}_{p_2,q_2}}.
  \]
\end{theorem}
By collecting the above results, we can prove
existence and properties of a Young integral.
\begin{proposition}\label{p:integral}
  Let $p_1,q_1,p_2\in[1,\infty)$, $q_2\in[1,\infty]$,
  $\alpha_1>0$, $\alpha_2\in\R$, and $r_1,r_2,r_3>0$
  be such that
  \[
    1 < \frac1{q_1} + \frac1{p_2} < 1 + \frac{\alpha_1+\alpha_2}{d},\qquad
    \frac1{r_1} + \frac1{r_2} > 1,\qquad
    \frac1{r_1} + \frac{\gamma}{r_3} > 1,
  \]
  for every $\gamma\in(0,1)$, $\gamma\leq\gamma_0$, with
  $\gamma_0=\alpha_1+\alpha_2-d\bigl(\frac1{q_1}+\frac1{p_2}-1\bigr)$.
  
  Let $\omega:[0,T]\to\R^d$ be a continuous path
  such that its occupation measure
  $\mu^\omega_{0,\cdot}\in C^\var{r_1}([0,T];\sbe^{\alpha_1,p_1}_{q_1})$,
  let $f\in C^\var{r_2}([0,T];B^{\alpha_2}_{p_2,q_2}(\R^d))$
  and $\theta\in C^\var{r_3}([0,T];\R^d)$.
  
  Then
  \[
    t\mapsto\int_0^t f(s,\theta_s-\omega_s)\,ds
  \]
  is well defined in $C^\var{r_1}([0,T];\R^d)$, and
  \[
    \begin{aligned}
      \lefteqn{\Bigl\|\int_0^\cdot f(s,\theta_s-\omega_s)\,ds\Bigr\|_{\var{r_1}}}\qquad&\\
        &\lesssim \bigl(\|f\|_{C^\var{r_2}(0,T;B^{\alpha_2}_{p_2,q_2})}
          +(1+\|\theta\|_{\var{r_3}}^\gamma)\|f\|_{L^\infty(0,T;B^{\alpha_2}_{p_2,q_2})}\bigr)
          \|\mu_{0,\cdot}^\omega\|_{C^\var{r_1}(0,T;\sbe^{\alpha_1,p_1}_{q_1})}.
    \end{aligned}
  \]
  
  Moreover, if $\theta^1,\theta^2\in C^\var{r_3}([0,T];\R^d)$, and if $\lambda\in(0,1)$
  is such that $\frac1{r_1}+\frac{\gamma(1-\lambda)}{r_3}>1$, then
  \begin{equation}\label{e:Icontinuous}
    \begin{aligned}
      \lefteqn{\Bigl\|\int_0^\cdot f(s,\theta^1_s-\omega_s)\,ds
          - \int_0^\cdot f(s,\theta^2_s-\omega_s)\,ds\Bigr\|_\var{r_1}\lesssim}\qquad&\\
        &\lesssim \|\mu_{0\cdot}^\omega\|_{C^\var{r_1}([0,T];\sbe^{\alpha_1,p_1}_{q_1})}
          \|f\|_{C^\var{r_2}([0,T];B^{\alpha_2}_{p_2,q_2})}
          \|\theta^1-\theta^2\|_\infty^{\lambda\gamma}\cdot\\
        &\quad\cdot\bigl(\|\theta^1-\theta^2\|_\infty + \|\theta^1\|_{C^\var{r_3}([0,T]:\R^d)}
           + \|\theta^2\|_{C^\var{r_3}([0,T];\R^d)}\bigr)^{\gamma(1-\lambda)}.
    \end{aligned}
  \end{equation}
  Finally, if $\gamma_0>1$, then the map
  \[
    \theta\to\Is(\theta)=\int_0^\cdot f(s,\theta_s-\omega_s)\,ds
  \]
  is Fréchet differentiable and, for $\gamma\in(0,1)$, with
  $\gamma\leq\gamma_0-1$,
  \begin{equation}\label{e:Idifferentiable}
    \begin{aligned}
      \|D\Is(\theta)\|
        &\lesssim \|\mu_{0\cdot}^\omega\|_{C^\var{r_1}([0,T]:\sbe^{\alpha_1,p_1}_{q_1})}\cdot\\
        &\qquad \cdot\bigl(\|f\|_{L^\infty([0,T];B^{\alpha_2}_{p_2,q_2})}
            + \|f\|_{C^\var{r_2}([0,T];B^{\alpha_2}_{p_2,q_2})}
            + \|\theta\|_{C^\var{r_3}([0,T];\R^d)}^{\gamma}\bigr).
    \end{aligned}
  \end{equation}
\end{proposition}
\begin{proof}
  Set
  \begin{equation}\label{e:chi}
    \chi_{st}
      = \int_s^t f(s,\theta_s-\omega_r)\,dr
      = f(s,\cdot)\star\mu^\omega_{st}(\theta_s),
  \end{equation}
  thus for $s\leq u\leq t$,
  \[
    \delta\chi_{sut}
      = \bigl(f(s,\cdot)-f(u,\cdot)\bigr)\star\mu^\omega_{ut}(\theta_s)
        + f(u,\cdot)\star\mu^\omega_{ut}(\theta_s)
          - f(u,\cdot)\star\mu^\omega_{ut}(\theta_u).
  \]
  Notice preliminary that by \autoref{besovregularity},
  $\mu^\omega_{ut}\in B^{\alpha_1}_{q_1,\infty}$,
  therefore by \autoref{t:convolution}
  $f(u,\cdot)\star\mu^\omega_{ut}\in B^{\alpha_1+\alpha_2}_{p,q_2}$,
  with $1+\frac1p=\frac1{q_1}+\frac1{p_2}$. By Sobolev's
  embeddings, $f(u,\cdot)\star\mu^\omega_{ut}\in B^\gamma_{\infty,q_2}$
  with $\gamma_0=\alpha_1+\alpha_2-\frac{d}{p}$ and
  \[
    \begin{multlined}[.9\linewidth]
      \|f(u,\cdot)\star\mu^\omega_{ut}\|_{B^{\gamma_0}_{\infty,q_2}}
        \lesssim \|f(u,\cdot)\star\mu^\omega_{ut}\|_{B^{\alpha_1+\alpha_2}_{p,q_2}}\lesssim\\
        \lesssim \|f(u,\cdot)\|_{B^{\alpha_2}_{p_2,q_2}}
          \|\mu_{ut}^\omega\|_{B^{\alpha_1}_{q_1,\infty}}
        \lesssim \|f(u,\cdot)\|_{B^{\alpha_2}_{p_2,q_2}}
          \|\mu_{ut}^\omega\|_{\sbe^{\alpha_1,p_1}_{q_1}}.
    \end{multlined}
  \]
  A similar estimate clearly holds for $(f(s,\cdot)-f(u,\cdot))\star\mu^\omega_{ut}$.
  
  Fix $\gamma\in(0,1)$, with $\gamma\leq\gamma_0$.
  We use the above inequalities to estimate $|\delta\chi_{sut}|$,
  \[
    \begin{aligned}
      |\delta\chi_{sut}|
        &\leq |\bigl(f(s,\cdot)-f(u,\cdot)\bigr)\star\mu^\omega_{ut}(\theta_s)|
        + |f(u,\cdot)\star\mu^\omega_{ut}(\theta_s)
          - f(u,\cdot)\star\mu^\omega_{ut}(\theta_u)|\\
        &\leq \|(f(s,\cdot)-f(u,\cdot))\star\mu^\omega_{ut}\|_{B^\gamma_{\infty,q_2}}
          + \|f(u,\cdot)\star\mu^\omega_{ut}\|_{B^\gamma_{\infty,q_2}}
            |\theta_s-\theta_u|^\gamma\\
        &\lesssim \rho_{su}^{1/r_0}\sigma_{ut}^{1/r_1},
    \end{aligned}
  \]
  where $F_T=\|f\|_{L^\infty(0,T;B^{\alpha_2}_{p_2,q_2})}$,
  $r_0=\max(r_2,r_3/\gamma)$, and
  \[
    \begin{aligned}
      \sigma_{st}
        &= \|\mu^\omega_{0,\cdot}\|^{r_1}_{C^\var{r_1}([s,t];\sbe^{\alpha_1,p_1}_{q_1})},\\
      \rho_{st}
        &= F_T^{r_0}\|\theta\|_{\var{r_3},[s,t]}^{\gamma r_0}
          + \|f\|^{r_0}_{C^\var{r_2}([s,t];B^{\alpha_2}_{p_2,q_2})},
    \end{aligned}
  \]
  are controls, and by assumption $\frac1{r_1}+\frac1{r_0}>1$.
  Moreover,
  \[
    |\chi_{st}|
      \lesssim \|f\|_{L^\infty(0,T;B^{\alpha_2}_{p_2,q_2})}
        \|\mu^\omega_{st}\|_{\sbe^{\alpha_1,p_1}_{q_1}}
      \leq F_T\sigma_{st}^{1/r_1}
  \]
  therefore by \autoref{l:sewing} the conclusion follows.

  We turn to the proof of \eqref{e:Icontinuous}. Given
  $\theta^1,\theta^2\in C^\var{r_3}([0,T];\R^d)$,
  it is not difficult to prove through \autoref{l:sewing} that
  \[
    \int_0^t f(s,\theta_s^1-\omega_s)\,ds - \int_0^t f(s,\theta_s^2-\omega_s)\,ds
      = \int_0^t \bigl(f(s,\theta_s^1-\omega_s)-f(s,\theta_s^2-\omega_s)\bigr)\,ds,
  \]
  and that the integral on the right hand side can be obtained by
  \autoref{l:sewing} with the local description
  \[
    \chi^0_{st}
      = f(s,\cdot)\star\mu_{st}^\omega(\theta_s^1)
        - f(s,\cdot)\star\mu_{st}^\omega(\theta_s^2).
  \]
  Moreover,
  \[
    \begin{aligned}
      \delta\chi^0_{sut}
        &= (f(s,\cdot)-f(u,\cdot))\star\mu^\omega_{ut}(\theta_s^1)
          - (f(s,\cdot)-f(u,\cdot))\star\mu^\omega_{ut}(\theta_s^2)\\
        &\quad  + f(u,\cdot)\star\mu^\omega_{ut}(\theta_s^1) - f(u,\cdot)\star\mu^\omega_{ut}(\theta_s^2)
          - f(u,\cdot)\star\mu^\omega_{ut}(\theta_u^1) + f(u,\cdot)\star\mu^\omega_{ut}(\theta_u^2)
      \end{aligned}
  \]
  can be estimated as
  \[
    \begin{aligned}
      |\delta\chi^0_{sut}|
        &\leq \|(f(s,\cdot)-f(u,\cdot))\star\mu^\omega_{ut}\|_{B^\gamma_{\infty,q_2}}
          |\theta_s^1-\theta_s^2|^\gamma + {}\\
        &\quad  + \|f(u,\cdot)\star\mu^\omega_{ut}\|_{B^\gamma_{\infty,q_2}}
          (|\theta_s^1-\theta_s^2|^\gamma + |\theta_u^1-\theta_u^2|^\gamma)\\
        &\lesssim \sigma_{ut}^{1/r_1}\bigl((\rho^f_{su})^{1/r_2}\|\theta^1-\theta^2\|_\infty^\gamma
          + F_T\|\theta^1-\theta^2\|_\infty^\gamma\bigr),
    \end{aligned}
  \]
  and as
  \[
    \begin{aligned}
      |\delta\chi^0_{sut}|
        &\leq (\rho^f_{su})^{1/r_2}\sigma_{ut}^{1/r_1}\|\theta^1-\theta^2\|_\infty^\gamma
          + \|f(u,\cdot)\star\mu^\omega_{ut}\|_{B^\gamma_{\infty,q_2}}
          (|\theta_s^1-\theta_u^1|^\gamma + |\theta_s^2-\theta_u^2|^\gamma)\\
        &\lesssim \sigma_{ut}^{1/r_1}\bigl((\rho^f_{su})^{1/r_2}\|\theta^1-\theta^2\|_\infty^\gamma
          + F_T(\rho^{\theta^1}_{su}+\rho^{\theta^2}_{su})^{\gamma/r_3}\bigr),
    \end{aligned}
  \]
  where $\rho^f_{st}=\|f\|^{r_2}_{\var{r_2},[s,t]}$, and similarly for
  $\rho^{\theta^i}$, $i=1,2$ in $C^\var{r_3}$.
  So for $\lambda\in(0,1)$,
  \[
    |\delta\chi^0_{sut}|
      \leq \sigma_{ut}^{1/r_1}\bigl((\rho^f_{su})^{1/r_2}\|\theta^1-\theta^2\|_\infty^\gamma
        + F_T\|\theta^1-\theta^2\|_\infty^{\gamma\lambda}
          (\rho^{\theta^1}_{su}+\rho^{\theta^2}_{su})^{(1-\lambda)\gamma/r_3}\bigr).
  \]
  Take this time $r_0=\max(r_2,r_3/(\gamma(1-\lambda)))$, and $\lambda$ small enough
  that $1/r_1+1/r_0>1$ still holds. Since we also have
  \[
    |\chi^0_{st}|
      \leq \|f(s,\cdot)\star\mu_{st}^\omega\|_{B^\gamma_{\infty,q_2}}
        \|\theta^1-\theta^2\|_\infty^\gamma
      \lesssim F_T\sigma_{st}^{1/r_1}\|\theta^1-\theta^2\|_\infty^\gamma,
  \]
  \eqref{e:Icontinuous} follows.
  
  We finally prove differentiability of $\Is$. Assume $\gamma_0>1$,
  and let now $\gamma\in(0,1)$ with $\gamma\leq \gamma_0-1$.
  First notice that if $\theta,\eta\in C^\var{r_3}(0,T;\R^d)$,
  the function $\chi'_{st}=\eta_s\cdot\nabla(f(s,\cdot)\star\mu_{st})(\theta_s)$
  meets the assumption of \autoref{l:sewing}, and thus defines
  an integral $\Js(\theta)\eta$, which is moreover linear in $\eta$.
  This follows as for the local description defined in \eqref{e:chi}.
  Indeed, $|\chi'_{st}|\leq \|\eta\|_\infty\|f\|_{L^\infty(0,T;B^{\alpha_2}_{p_2,q_2})}
  \|\mu_{st}^\omega\|_{\sbe^{\alpha_1,p_1}_{q_1}}$, and
  \[
    \begin{aligned}
      \delta\chi_{sut}'
        &= \eta_s\cdot\nabla\bigl((f(s,\cdot)-f(u,\cdot))\star\mu_{ut}^\omega\bigr)(\theta_s)
          + (\eta_s-\eta_u)\cdot\nabla\bigl(f(u,\cdot)\star\mu_{ut}^\omega\bigr)(\theta_s)\\
        &\quad + \eta_u\cdot\bigl(\nabla\bigl(f(u,\cdot)\star\mu_{ut}^\omega\bigr)(\theta_s)
            - \nabla\bigl(f(u,\cdot)\star\mu_{ut}^\omega\bigr)(\theta_u)\bigr),
    \end{aligned}
  \]
  thus
  \[
    |\delta\chi_{sut}'|
      \leq\sigma_{ut}^{1/r_1}\bigl(\|\eta\|_\infty(\rho^f_{su})^{1/r_2}
        + F_T(\rho^\eta_{su})^{1/r_3}
        + \|\eta\|_\infty F_T (\rho^\theta_{su})^{\gamma/r_3}\bigr)
  \]
  where $\sigma$, $F_T$, $\rho^f$, $\rho^\theta$ are as before,
  and $\rho^\eta$ is defined as $\rho^\theta$.
  Let us prove that $\Is(\theta)$ is differentiable and
  $D\Is(\theta)=\Js(\theta)$. By our previous considerations,
  given $\theta,\eta\in C^\var{r_3}([0,T];\R^d)$ and
  $\epsilon>0$, the integral
  \[
    \tfrac1\epsilon\bigl(\Is(\theta+\epsilon\eta) - \Is(\theta)\bigr) - \Js(\theta)\eta
  \]
  has the local description
  \[
    \begin{aligned}
      \chi_{st}^\epsilon
        &= \tfrac1\epsilon\bigl( (f(s,\cdot)\star\mu_{st}^\omega)(\theta_s+\epsilon\eta_s)
          - (f(s,\cdot)\star\mu_{st}^\omega)(\theta_s)\bigr)
          - \eta_s\cdot\nabla(f(s,\cdot)\star\mu_{st}^\omega)(\theta_s)\\
        &= \frac1\epsilon\int_0^\epsilon\eta_s\cdot\bigl(
          \nabla(f(s,\cdot)\star\mu_{st}^\omega)(\theta_s+a\eta_s)
          - \nabla(f(s,\cdot)\star\mu_{st}^\omega)(\theta_s)\bigr)\,da
    \end{aligned}
  \]
  Now,
  \[
    |\chi_{st}^\epsilon|
      \leq\frac1\epsilon\int_0^\epsilon a^\gamma|\eta_s|^{\gamma+1}
        \|f(s,\cdot)\star\mu_{st}^\omega\|_{B^{\gamma+1}_{\infty,q_2}}\,da
      \lesssim\epsilon^\gamma\|\eta\|_\infty^{\gamma+1}
        F_T 
         \|\mu_{st}^\omega\|_{\sbe^{\alpha_1,p_1}_{q_1}}
  \]
  and $\delta\chi_{sut}^\epsilon=\cerchiato{\small a}+\cerchiato{\small b}
  +\cerchiato{\small c}$, where
  \[
    \begin{aligned}
      |\cerchiato{\small a}|
        &= \Big|\frac1\epsilon\int_0^\epsilon(\eta_s-\eta_u)\cdot\bigl(
          \nabla(f(s,\cdot)\star\mu_{ut}^\omega)(\theta_s+a\eta_s)
          - \nabla(f(s,\cdot)\star\mu_{ut}^\omega)(\theta_s)\bigr)\,da\Big|\\
        &\lesssim \epsilon^\gamma F_T\|\eta\|_\infty^\gamma
          \sigma_{ut}^{1/r_1}(\rho^\eta_{su})^{1/r_3},
    \end{aligned}
  \]
  \[
    \begin{aligned}
      \cerchiato{\small b}
        &=\Big|\frac1\epsilon\int_0^\epsilon\eta_u\cdot\bigl(
          \nabla((f(s,\cdot)-f(u,\cdot))\star\mu_{ut}^\omega)(\theta_s+a\eta_s)\\
        &\qquad - \nabla((f(s,\cdot)-f(u,\cdot))\star\mu_{ut}^\omega)(\theta_s)\bigr)\,da\Big|\\
        &\lesssim \epsilon^\gamma\|\eta\|_\infty^{\gamma+1}
          \sigma_{ut}^{1/r_1}(\rho^f_{su})^{1/r_2},
    \end{aligned}
  \]
  and finally,
  \[
    \begin{multlined}[.9\linewidth]
      \cerchiato{\small c}
        =\frac1\epsilon\int_0^\epsilon\eta_u\cdot\Bigl[\bigl(
          \nabla(f(u,\cdot)\star\mu_{ut}^\omega)(\theta_s+a\eta_s)
          - \nabla(f(u,\cdot)\star\mu_{ut}^\omega)(\theta_s)\bigr)\\
          - \bigl(\nabla(f(u,\cdot)\star\mu_{ut}^\omega)(\theta_u+a\eta_u)
          - \nabla(f(u,\cdot)\star\mu_{ut}^\omega)(\theta_u)\bigr)\Bigr]
          \,da
    \end{multlined}
  \]
  can be estimated in two different ways (as we have already done previously
  for similar terms) as
  $\cerchiato{\small c}\lesssim\epsilon^\gamma\|\eta\|_\infty^{\gamma+1}F_T\sigma_{ut}^{1/r_1}$
  and
  $\cerchiato{\small c}\lesssim F_T\|\eta\|_\infty\sigma_{ut}^{1/r_1}
  (\rho^\theta_{su}+\rho_{su}^\eta)^{\gamma/r_3}$, so that for $\lambda\in(0,1)$,
  \[
    \cerchiato{\small c}
      \lesssim \epsilon^{\lambda\gamma}F_T\|\eta\|_\infty^{1+\lambda\gamma}
        \sigma_{ut}^{1/r_1}(\rho^\theta_{su}+\rho_{su}^\eta)^{\frac\gamma{r_3}(1-\lambda)}.
  \]
  In conclusion,
  \[
    \bigl\|\tfrac1\epsilon\bigl(\Is(\theta+\epsilon\eta) - \Is(\theta)\bigr)
        - \Js(\theta)\eta\bigr\|_\var{r_3}
      \lesssim\epsilon^{\lambda\gamma}.    
  \]
  Moreover, the above estimate is uniform for $\eta$ in a bounded ball
  centred at $0$. This proves Fréchet differentiability of $\Is(\theta)$.
\end{proof}
We turn to ODEs driven by \sbe-paths.
Following \cite{GalGub2020a},
we shall use the
integral as given by the previous proposition
to give a meaning, for general drifts, to differential
equations.
\begin{definition}\label{d:ode}
  Given $x_0\in\R^d$, $r_2\geq1$, $\alpha_2\in\R$,
  $p_2\in[1,\infty)$, $q_2\in[1,\infty]$, and
  $f\in C^\var{r_2}([a,b];B^{\alpha_2}_{p_2,q_2}(\R^d))$,
  a map $x\in C([a,b];\R^d)$ is a solution of the ODE
  \[
    x_t
      = x_0 - \omega_t + \int_a^t f(s,x_s)\,ds,
      \qquad t\in[a,b],
  \]
  if there are $p_1,q_1\in[1,\infty)$, $\alpha_1>0$
  and $r_1\geq1$ such that 
  \begin{equation}\label{e:ode}
    1 < \frac1{q_1} + \frac1{p_2} < 1 + \frac{\alpha_1+\alpha_2}{d},\qquad
    \frac1{r_1} + \frac1{r_2} > 1,\qquad
    r_1 < 1 + \gamma,
  \end{equation}
  for some $\gamma\in(0,1)$ and $\gamma<\gamma_0$, with
  $\gamma_0=\alpha_1+\alpha_2-d\bigl(\frac1{q_1}+\frac1{p_2}-1\bigr)$,
  and $\theta_\cdot=x_\cdot+\omega_\cdot$ has occupation measure
  $\mu^\omega_{a,\cdot}\in C^\var{r_1}([a,b];\sbe^{\alpha_1,p_1}_{q_1})$,
  and if
  \[
    \theta_t
      = x_0 + \int_0^tf(s,\theta_s-\omega_s)\,ds,
      \qquad t\in[a,b].
  \]
\end{definition}
\begin{theorem}\label{t:ode}
  Let $p_1,q_1,p_2\in[1,\infty)$, $q_2\in[1,\infty]$,
  $\alpha_1>0$, $\alpha_2\in\R$, and $r_1,r_2>0$
  be such that \eqref{e:ode} holds
  for some $\gamma\in(0,1)$ and $\gamma\leq\gamma_0$,
  with $\gamma_0=\alpha_1+\alpha_2-d\bigl(\frac1{q_1}+\frac1{p_2}-1\bigr)$.
  
  Let $x_0\in\R^d$, $\omega:[0,T]\to\R^d$ be a continuous path
  such that its occupation measure
  $\mu^\omega_{0,\cdot}\in C^\var{r_1}([0,T];\sbe^{\alpha_1,p_1}_{q_1})$,
  and let $f\in C^\var{r_2}([0,T];B^{\alpha_2}_{p_2,q_2}(\R^d))$.

  Then there is $x\in C^\var{r_1}([0,T];\R^d)$ such that
  \[
    x_t
      = x_0 - \omega_t + \int_0^t f(s,x_s)\,ds,
  \]
  for $0\leq t\leq T$, in the sense of \autoref{d:ode}.
  
  Moreover, if $\gamma_0>1$ then the solution is unique.
\end{theorem}
\begin{proof}
  We look for a solution $x_t=\theta_t - \omega_t$, with
  $\theta$ a fixed point of the map
  $\Tc:K_R\to C^\var{r_3}([0,T_\star];\R^d)$, where
  \[
    (\Tc\theta)_t
      = x_0 + \int_0^t f(s,\theta_s-\omega_s)\,ds,
  \]
  $K_R=\{\theta\in C^\var{r_3}([0,T_\star];\R^d):\|\theta\|_\var{r_3}\leq R\}$,
  $r_3>r_1$ such that $\frac1{r_1}+\frac\gamma{r_3}>1$,
  and the integral is defined by \autoref{p:integral}.
  
  We use Schauder's fixed point. To this end, we shall prove that
  \begin{itemize}
    \item $\Tc$ maps $K_R$ onto and is continuous,
    \item $\Tc(K_R)$ is relatively sequentially compact.
  \end{itemize}
  First, from \autoref{p:integral}, for $r_3>r_1'>r_1$ but
  such that $\frac1{r_1'}+\frac\gamma{r_3}>1$
  and \eqref{e:ode} holds with $r_1$ replaced
  by $r_1'$, 
  \[
    \begin{aligned}
      \|\Tc\theta\|_\var{r_3}
        &\leq \|\Tc\theta\|_\var{r_1}\\
        &\begin{multlined}
 \leq (\|f\|_{C^\var{r_2}([0,T];B^{\alpha_2}_{p_2,q_2})}
          \\ + (1+\|\theta\|_\var{r_3})\|f\|_{L^\infty(0,T;B^{\alpha_2}_{p_2,q_2})})
          \|\mu_{0,\cdot}^\omega\|_{C^\var{r_1'}(0,T_\star;\sbe^{\alpha_1,p_1}_{q_1})}        
        \end{multlined}
    \end{aligned}
  \]
  and by assumption
  $\|\mu_{0,\cdot}^\omega\|_{C^\var{r_1'}(0,T_\star;\sbe^{\alpha_1,p_1}_{q_1})}\to 0$
  as $T_\star\to 0$. Therefore it is possible to find, given $R>0$,
  a time $T_\star\in(0,T]$ that depends only on $f$ and $\mu^\omega$,
  but not on $x_0$,
  such that $\Tc$ maps $K_R$ onto. Moreover, the same estimate proves
  that $\Tc K_R$ is a subset of $C^\var{r_1}([0,T_\star];\R^d)$

  To prove sequential compactness, it suffices by \cite[Proposition 5.28]{FriVic2010}
  to prove equicontinuity. Indeed, if $\theta\in K_R$ and $r_1'>r_1$ is as above,
  \[
    |(\Tc\theta)_t - (\Tc\theta)_s|
      = \Bigl|\int_s^t f(u,\theta_u-\omega_u)\,du\Bigr|
      \lesssim \|\mu_{0,\cdot}^\omega\|_{C^\var{r_1'}(s,t;\sbe^{\alpha_1,p_1}_{q_1})}
  \]
  where the constant depends only in $R$ and $f$. Therefore elements in $\Tc(K_R)$
  have the same modulus of continuity and equicontinuity holds. Finally,
  continuity of $\Tc$ follows from \eqref{e:Icontinuous}.

  Finally, we construct a solution on the whole interval $[0,T]$. As noticed before,
  the existence time $T_\star$ we have found depends on $R$,
  $\|f\|_{C^\var{r_2}([0,T];B^{\alpha_2}_{p_2,q_2})}$
  and $\|\mu_{0,T_\star}^\omega\|_{C^\var{r_1'}([0,T_\star],\sbe^{\alpha_1,p_1}_{q_1})}$,
  with $r_1'>r_1$, sufficiently close to $r_1$. So the global solution
  is obtained by gluing the solution $x^1$ in $[0,T_\star^1]$, with $T_\star^1=T_\star$,
  started at $x_0$, with the solution $x^2$ on $[T_\star^1,T_\star^2]$ with initial
  condition $x^2(T_\star^1)=x^1(T_\star^1)$, etc. If $T_\star^0=0$, the existence
  times $(T^k_\star)_{k\geq0}$ can be chosen so that
  \[
    (\|f\|_{C^\var{r_2}([0,T];B^{\alpha_2}_{p_2,q_2})}
        + (1+R_k)\|f\|_{L^\infty(0,T;B^{\alpha_2}_{p_2,q_2})})
        \|\mu_{0,\cdot}^\omega\|_{C^\var{r_1'}(T_\star^{k-1},T_\star^k;\sbe^{\alpha_1,p_1}_{q_1})}
      \geq \frac{R_k}2,
  \]
  for $R_k>0$ chosen in the proof. Let $R_k=1$, then in a finite number
  of steps the whole $[0,T]$ is covered (that is $T_\star^k=T$ for some $k$).
  Indeed, we have that 
  $\|\mu_{0,\cdot}^\omega\|_{C^\var{r_1'}(T_\star^{k-1},T_\star^k;\sbe^{\alpha_1,p_1}_{q_1})}
  \geq C_f$, for a constant depending only on $f$, thus
  \[
    \sum_{k=1}^N C_F^{r_1'}
      \leq \sum_{k=1}^N \|\mu_{0,\cdot}^\omega\|_
        {C^\var{r_1'}(T_\star^{k-1},T_\star^k;\sbe^{\alpha_1,p_1}_{q_1})}^{r_1'}
      \leq \|\mu_{0,\cdot}^\omega\|_
              {C^\var{r_1'}(0,T;\sbe^{\alpha_1,p_1}_{q_1})}^{r_1'},
  \]
  which concludes the proof of global existence.
  
  Finally, uniqueness holds by using the Banach fixed point theorem
  and the estimate on the differential of the integral given
  in \autoref{p:integral}.
\end{proof}
\begin{remark}
  The global control on $f$ yields a global solution.
  If on the other hand $f\in C^\var{r_2}([0,T];B^{\alpha_2}_{p_2,q_2}(D))$
  for a domain $D$, or $f\in C^\var{r_2}([0,T];B^{\alpha_2}_{p_2,q_2,\text{loc}})$,
  the previous theorem provides a local solution, namely there is
  a time $T_\star\in(0,T]$ and a solution $x$ defined on $[0,T]$.
  Indeed if $\eta\in C^\infty$ is a cut-off function such that
  $\eta\equiv1$ in a neighbourhood of $x_0$, then
  in both cases $\eta f\in C^\var{r_2}([0,T];B^{\alpha_2}_{p_2,q_2})$,
  and the previous theorem applies.
\end{remark}
Recall that $\phi:\{0\leq s\leq t\leq T\}\times\R^d\to\R^d$ is
a (continuous) flow if
\begin{itemize}
  \item $\phi(t,t,x_0)=x_0$, for all $t,x_0$,
  \item $\phi(u,t,\phi(s,u,x_0))=\phi(s,t,x_0)$ for all $s\leq u\leq t$ and $x_0$,
  \item $\phi(s,t,\cdot)$ is continuous with continuous inverse,
    for all $s,t$.
\end{itemize}
\begin{corollary}
  Under the same assumptions of \autoref{t:ode}, if $\gamma_0>1$,
  then there is a flow $\phi$ of diffeomorphisms (namely, $\phi(s,t,\cdot)$ is
  continuously differentiable) such that $\phi(s,\cdot,x_0)\in C^\var{r_1}(s,T;\R^d)$
  for all $s,x_0$, and
  \begin{equation}\label{e:flow}
    \phi(s,t,x_0)
      = x_0 - (\omega_t - \omega_s)
        + \int_s^t f(r,\phi(s,r,x_0))\,dr,
  \end{equation}
  for all $s\leq t,x_0$.
\end{corollary}
\begin{proof}
  Define $\phi$ according to \eqref{e:flow}, using the previous theorem.
  The first property of flows is obvious, the second follows by uniqueness.
  The fact that $\phi(s,\cdot,x_0)\in C^\var{r_1}(s,T;\R^d)$ follows by
  definition.

  To prove continuity, recall that if
  $\theta(s,t,x_0)=\phi(s,t,x_0) + (\omega_t - \omega_s)$, then
  \[
    \theta(s,t,x_0)
      = x_0
        + \int_s^t f(r,\theta(s,r,x_0)-(\omega_r-\omega_s))\,dr,
  \]
  therefore, by \eqref{e:Icontinuous}, for $u\in[s,t]$,
  \[
    \begin{multlined}[.9\linewidth]
      \|\theta(s,\cdot,x_0) - \theta(s,\cdot,x_0')\|_{C^\var{r_1}([s,u])}\leq\\
        \leq |x_0-x_0'|
          + c_u\|\theta(s,\cdot,x_0) - \theta(s,\cdot,x_0')\|_{C^\var{r_1}([s,u])}^a,
    \end{multlined}
  \]
  where $a\in(0,1)$ and
  \[
    \begin{multlined}[.9\linewidth]
      c_u
        \lesssim \|\mu^\omega_{s,\cdot}\|_{C^\var{r_1'}([s,u])}
          \|f\|_{C^\var{r_2}([s,t])}\cdot\\
          \cdot(\|\theta(s,\cdot,x_0)\|_{C^\var{r_1}([s,t])} + \|\theta(s,\cdot,x_0')\|_{C^\var{r_1}([s,t])})^b
    \end{multlined}
  \]
  with $r_1'>r_1$ and $b<1$. Since $c_u\to0$ as $u\to s$, we can find a value
  $u_1$ such that $\|\theta(s,\cdot,x_0) - \theta(s,\cdot,x_0')\|_{C^\var{r_1}([s,u_1])}
  \leq 2|x_0-x_0'|$. Since the estimate on $c_u$ depends only on $\mu^\omega$,
  we can replicate the estimate on a at most finite number of intervals
  (as in the proof of \autoref{t:ode}) and get continuity on $[s,t]$.
  
  To prove that $\phi$ is invertible with continuous inverse, notice that,
  for fixed $s,t$, if we define the time-reversed flow
  \[
    \psi(s,r,y_0)
      = y_0 + (\omega_{t+s-r} - \omega_t)
        - \int_s^r f(t+s-u,\psi(s,u,y_0))\,du,
  \]
  with $r\in[s,t]$, then $\psi(s,t,\phi(s,t,x_0))=x_0$.
  The flow $\psi$ is also defined according to
  \autoref{t:ode}, and it is immediate to check
  that the reversed perturbation $r\mapsto\omega_{t+s-r}$
  is in $C^\var{r_1}([s,t];\sbe^{\alpha_1,p_1}_{q_1})$.
  Finally, differentiability follows by an argument similar to continuity,
  using this time \eqref{e:Idifferentiable}.
\end{proof}
\section{Examples}\label{s:examples}

Finally, we discuss a few examples of processes with occupation measure
in \sbe-type spaces. First, using either (especially) \autoref{cor:fbHSBE}
or \autoref{cor:fbHSBE-m}, it is not difficult, in general, to
establish that Gaussian processes have occupation measure in
\sbe spaces. We provide a pair of examples of non-Gaussian
processes with occupation measure in \sbe spaces. The two
examples are solutions of stochastic differential equations
driven either by Brownian motion or fractional Brownian
motion. Albeit this is not surprising, since it is expected
that short time asymptotics of densities of solutions of
SDEs look like the rough driving process, it remains
non obvious to prove similar statements (namely that
solutions of equation with \sbe-driven input have \sbe
occupation measures) in the general framework
of Young differential equations (\autoref{s:young}).
\subsection{Markov processes and one dimensional SDEs}

Consider a Markov process $(X_t)_{t\in[0,T]}$ on $R^d$
with transition density function $p(s,t,x,y)$, and
assume $X_0=x_0\in\R^d$. Given $n\geq1$ and
$s_0=0<s_1<s_2<\dots<s_n$, the joint density of
$(X_{s_1},X_{s_2},\dots,X_{S_n})$ is
\begin{equation}\label{e:markovdensity}
  \rho_{s_{1:n}}(x_{1:n})
    = \prod_{j=1}^n p(s_{j-1},s_j,x_{j-1},x_j)
\end{equation}
and the density of an increment is
\[
  \nu_{s_1s_2}(y)
    = \int_{\R^d}\rho_{s_1s_2}(x,x+y)\,dx
    = \int_{\R^d} p(0,s_1,x_0,x)p(s_1,s_2,x,x+y)\,dx.
\]
Therefore,
\[
  \|\nu_{s_1s_2}\|_{C^\beta}
    \leq \int_{\R^d}\|p(s_1,s_2,x,\cdot)\|_{C^\beta}p(0,s_1,x_0,x)\,dx,
\]
and, for instance, \autoref{thm:spSBE-m} applies with $m=1$
if there
are $C_\nu>0$ and $\eta\in(0,1)$ such that for all
$0\leq a<b\leq T$,
\[
  \int_a^b\int_{s_1}^b\sup_{x\in\R^d}\|p(s_1,s_2,x,\cdot)\|_{C^\beta}\,ds_2
    \leq C_\nu(b-a)^{2\eta}.
\]
If on the other hand we wish to use \autoref{thm:spSBE-m} with $m\geq2$,
we notice that, to use for instance \eqref{eqn:Crho-m},
in view of \eqref{e:markovdensity} it is
sufficient to get a bound of
\begin{equation}\label{e:sde-m}
  \|(1-\Delta_x)^{\frac\beta2}(1-\Delta_y)^{\frac\beta2}p(s,t,x,y)\|_{L^\infty}.
\end{equation}

As an example for this, consider a stochastic equation
in dimension $1$,
\[
  \begin{cases}
    dX_t
      = b(t,X_t)\,dt + \sigma(t,X_t)\,dB_t,\\
    X_0
      = x_0,
  \end{cases}
\]
and let $p(s,t,x,\cdot)$ be the density of $X_t$ subject to $X_s=x$.
A simple method for the existence and regularity of the density
of the solution can be found in
\cite{Rom2018}. Short time asymptotics are classical, see
for instance \cite{Var1967a,Var1967b,Mol1975}. Assume for
simplicity $\sigma\equiv1$ (otherwise, one gets the
same asymptotics by \cite{Var1967b}) and $b$ bounded
measurable, then
\[
  p(s,t,x,y)
    = q_{t-s}(x-y) + \int_s^t \partial_y q_{t-r}\star\bigl(b(r,\cdot)p(s,r,x,\cdot)\bigr)\,dr,
\]
where $q_t$ is the heat kernel. Then
\[
  \|p(s,t,x,\cdot)\|_{C^\beta}
    \leq \|q_{t-s}\|_{C^\beta}
      + \|b\|_{L^\infty}\int_s^t \|q_{t-r}\|_{W^{1+\beta,1}}\|p(s,r,x,\cdot)\|_{C^\beta}\,dr,
\]
so that by the heat kernel asymptotics
$\|q_t\|_{C^\beta}\sim t^{-\frac12(1+\beta)}$, $\|q_t\|_{W^{\beta,1}}\sim t^{-\frac\beta2}$
and a simple Gronwall's argument
(using for instance \cite[Lemma 3.3]{MouWeb2017})
yields
\[
  \|p(s,t,x,\cdot)\|_{C^\beta}
    \lesssim (t-s)^{-\frac12(1+\beta)},
\]
with a pre-factor independent from $x$. Thus \autoref{thm:spSBE-m}
applies with $m=1$ and $\eta=\frac12(1-\beta)$.

Likewise, if we wish to use \autoref{thm:spSBE-m} with $m\geq 2$, we
get an estimate of \eqref{e:sde-m} notice
that to use, for instance, \eqref{eqn:Crho-m}, in view of
\eqref{e:markovdensity} it is sufficient to bound
\[
  \|(1-\Delta_x)^{\frac\beta2}(1-\Delta_y)^{\frac\beta2}p(s,t,\cdot,\cdot)\|_\infty
\]
and with similar computations and Schauder estimates
we get
\[
  \|(1-\Delta_x)^{\frac\beta2}(1-\Delta_y)^{\frac\beta2}p(s,t,\cdot,\cdot)\|_\infty
    \lesssim (t-s)^{-\frac12(1+2\beta)}
\]
for $\beta$, so that
\[
  \|(1-\Delta_{x_1})^{\frac\beta2}\dots(1-\Delta_{x_{2m}})^{\frac\beta2}
      \rho{s_{1:2m}}\|_\infty
    \lesssim s_1^{-\frac12(1+\beta)}\prod_{i=1}^{2m-1}(s_{i+1}-s_i)^{-\frac12(1+2\beta)},
\]
and the occupation measure of $(X_t)_{t\in[0,T]}$ is in
$C^\var{p}([0,T];\sbe^{\alpha,2m}_{2m})$ for $\alpha<1$ and
$p>\frac{4m}{(2m-1)(1-\alpha)}$. A smaller value of $p$
can be obtained on intervals away from zero, or using
\autoref{thm:spSBE-m} instead of
\autoref{prop:spSBE-m}.
\subsection{Fractional Brownian motion}

From \autoref{cor:fbHSBE} and \autoref{t:ode} we can immediately deduce
the following result.
\begin{corollary}
  Let $(B^H_t)_{t\geq0}$ be a $d$-dimensional fractional Brownian motion
  with Hurst index $H\in(0,1)$, with $H<\frac1d$, and $x_0\in\R^d$.
  Then if $\alpha_1<\frac1{2H}-\frac{d}{2}$ and
  $f\in C^\var{r_2}([0,T];B^{\alpha_2}_{p_2,q_2})$, with $r_2<2$,
  $p_2<2$, and $\alpha_1+\alpha_2>1+\frac{d}{p_2}-\frac{d}{2}$,
  then there is $x\in C([0,T];\R^d)$ such that
  \[
    x_t
      = x_0 + B_t^H + \int_0^t f(s,x_s)\,ds,
      \qquad t\leq T.
  \]
\end{corollary}
From \cite[Lemma 7.1]{Pit1978} we see that fractional Brownian motion
with Hurst index $H$ is strongly local $\phi$ non deterministic,
with $\phi(r)=r^{2H}$. As noticed by \cite{Xia2006}, by modifying
the proof of \cite[Proposition 7.2]{Pit1978}, this yields
\eqref{e:lnd} on time intervals that are bounded away
from $0$. Thus the occupation measure of fractional
Brownian motion is in $C^\var{p}([a,b];B^{\alpha}_{q,\infty})$
for $0<a<b$ and $\alpha,p,q$ as in \eqref{e:lndparams}.

We turn to regularity of the occupation measure of solutions
of equations driven by a fractional Brownian motion. We
consider the setting of \cite{BauNuaOuyTin2016}, namely
we consider
\begin{equation}\label{e:fsde}
  X_t
    = x + \int_0^t V_0(X_s)\,ds
      + \sum_{i=1}^d\int_0^t V_i(X_s)\,dB_s^i,
\end{equation}
where $x\in\R^m$, $(B^1,\dots,B^d)$ is a $d$-dimensional
fractional Brownian motion with Hurst parameter
$H>\frac14$, and $V_0,V_1,\dots,V_d$ are smooth vector fields
on $\R^m$, bounded with all derivatives bounded.
Moreover, $(V_1,\dots,V_d)$ are elliptic and
non-degenerate (see \cite[Hypotheses 1.2,1.3]{BauNuaOuyTin2016}).
Existence of the above equation has been established
for instance in \cite{FriVic2010}, while existence and
regularity of the density has been proved for instance
in \cite{BauHai2007,HuNua2007,NuaSau2009} ($H>\frac12$)
and \cite{CasLitLyo2013,CasHaiLitTin2015} ($H<\frac12$).
\begin{proposition}
  Let $d\leq 3$ and $H\in(\frac14,1)$ with $H<\frac1d$.
  Under the above assumptions, let $(X_t)_{t\in[0,T]}$
  be a solution of \eqref{e:fsde}, and
  $\mu^X$ its occupation measure.
  Then for every $a>0$, $\alpha>0$,
  with $\alpha<\frac1{2H}-\frac{d}{2}$,
  $\mu^X_{a,\cdot}\in V^p([a,T];\sbe^{\alpha,2}_2)$
  for all $p>1-\frac12H(d+2\alpha)$.
\end{proposition}

\begin{proof}
  Fix $\epsilon\leq s<t$, and let $\nu_{s,t}$ be the density of
  $X_t-X_s$. For an integer $k\geq1$, let
  $\alpha'\in\{1,\dots,d\}^k$ be a multi-index,
  and we write $\partial_{\alpha'}
  =\partial_{x_{\alpha'_1}}\dots\partial_{x_{\alpha'_k}}$.
  Let $\alpha=\alpha'\cup(1,2,\dots,k)$, so that
  $\partial_{\alpha}$ computes $k+d$ derivatives.
  Let $\varphi:\R^d\to\R$ be a smooth function,
  bounded by $1$,
  then by \cite[Proposition 2.1.4]{Nua2006},
  integration by parts holds and there is
  a random variable $H_{\alpha,s,t}$ such that
  \[
    \E[\partial_{\alpha}\varphi(X_t-X_s)]
      = \E[\varphi(X_t-X_s)H_{\alpha,s,t}].
  \]
  Moreover,
  \[
    \E[H_{\alpha,s,t}^2]^\frac12
      \lesssim \|\Gamma_{s,t}^{-1}D(X_t-X_s)\|_{k+d,2^{k+d+1}}^{k+d},
  \]
  where $D(X_t-X_s)$ is the Malliavin derivative,
  $\Gamma_{s,t}$ is the Malliavin covariance matrix
  of $X_t-X_s$, and $\|\cdot\|_{k,p}$ is the
  Malliavin-Sobolev norm (see \cite[Section 2.1]{Nua2006}).
  We thus have
  \[
    |\E[\partial_{\alpha}\varphi(X_t-X_s)]|
      \lesssim \|\Gamma_{s,t}^{-1}\|_{k+d,2^{k+d+2}}^{k+d}
        \|D(X_t-X_s)\|_{k+d,2^{k+d+2}}^{k+d}.
  \]
  As in \cite[Proposition 5.9]{BauNuaOuyTin2016} (the proposition
  proves only the case $k=0$, conditional to $s$,
  but the case $k>0$ can be carried on as in Lemma 4.1, Lemma 4.2
  of the same paper),
  \[
    \|D(X_t-X_s)\|_{k+d,2^{k+d+2}}
      \lesssim (t-s)^H,
        \qquad
    \|\Gamma_{s,t}^{-1}\|_{k+d,2^{k+d+2}}
      \lesssim (t-s)^{-2H}
  \]
  with constants depending on $k$ and $\epsilon$.
  In conclusion,
  $|\E[\partial_{\alpha}\varphi(X_t-X_s)]|\lesssim(t-s)^{-H(k+d)}$.
  To prove an estimate of the $k^\text{th}$ derivatives
  of $\nu_{s,t}$, we use a sequence of smooth functions
  that converge, monotonically increasing, to suitable
  Heaviside functions. This yields
  $\|\nu_{s,t}\|_{C^k}\lesssim(t-s)^{-H(k+d)}$.
\end{proof}
\bibliographystyle{amsalpha}

\end{document}